\documentclass[11pt]{article}
\usepackage{amsmath,amssymb,amscd, float, verbatim, latexsym,cite,wrapfig,enumerate,setspace,mathtools,bm}
\usepackage[top=1in, bottom=1in, left=1in, right=1in]{geometry}
\usepackage{amsthm,hyperref,amstext}
\usepackage[all]{xy}
\usepackage{lscape}
\usepackage{color}
\usepackage[normalem]{ulem}
\usepackage{cancel,float}
\usepackage{subcaption}
\usepackage{tikz-cd}
\usepackage{mathdots}
\mathtoolsset{showonlyrefs=true}

\newtheorem{theorem}{Theorem}[section]
\newtheorem{lemma}[theorem]{Lemma}
\newtheorem{proposition}[theorem]{Proposition}

\newtheorem{remark}[theorem]{Remark}
\newtheorem{definition}[theorem]{Definition}

\newcommand{\D}{0}
\newcommand{\bD}{\mathbb{D}}
\newcommand{\R}{\mathbb{R}}
\newcommand{\C}{\mathbb{C}}
\newcommand{\N}{{\mathbb{N}_{0}}}

\newcommand{\Z}{\mathbb{Z}}

\newcommand{\cA}{\mathcal{A}}

\newcommand{\cG}{\mathcal{G}}

\newcommand{\cM}{\mathcal{M}}

\newcommand{\cO}{\mathcal{O}}
\newcommand{\cQ}{\mathcal{Q}}

\newcommand{\bydef}{\stackrel{\mbox{\tiny\textnormal{\raisebox{0ex}[0ex][0ex]{def}}}}{=}}

\title{Validated matrix multiplication transform for orthogonal polynomials with applications to computer-assisted proofs for PDEs} 

\author{%
Matthieu Cadiot\thanks{%
	Department of Mathematics and Statistics, McGill University,
	Montreal, QC, H3A 0B9, Canada (\texttt{matthieu.cadiot@mail.mcgill.ca})}
\and
Jonathan~Jaquette\thanks{%
	New Jersey Institute of Technology, University Heights, Cullimore Hall 606, Newark, New Jersey, 07102, USA
    (\texttt{jonathan.jaquette@njit.edu})}
\and
Jean-Philippe~Lessard\thanks{%
	Department of Mathematics and Statistics, McGill University,
	Montreal, QC, H3A 0B9, Canada (\texttt{jp.lessard@mcgill.ca})}
\and
Akitoshi~Takayasu\thanks{
	Institute of Systems and Information Engineering, University of Tsukuba, 1-1-1 Tennodai, Tsukuba, Ibaraki 305-8573, Japan (\texttt{takitoshi@risk.tsukuba.ac.jp})} 
}

\begin{document}

\maketitle

\begin{abstract}
In this paper, we achieve three primary objectives related to the rigorous computational analysis of nonlinear PDEs posed on complex geometries such as disks and cylinders. First, we introduce a validated Matrix Multiplication Transform (MMT) algorithm, analogous to the discrete Fourier transform, which offers a reliable framework for evaluating nonlinearities in spectral methods while effectively mitigating challenges associated with rounding errors. Second, we examine the Zernike polynomials, a spectral basis well-suited for problems on the disk, and highlight their essential properties. We further demonstrate how the MMT approach can be effectively employed to compute the product of truncated Zernike series, ensuring both accuracy and efficiency. Finally, we combine the MMT framework and Zernike series to construct computer-assisted proofs that establish the existence of solutions to two distinct nonlinear elliptic PDEs on the disk. 
\end{abstract}

\begin{center}
\end{center}

\section{Introduction} \label{sec:Introduction}

The study of partial differential equations (PDEs) has long been a central topic in mathematics and applied sciences. Many of these equations, especially nonlinear ones, do not admit closed-form solutions, making their analysis and numerical resolution essential for practical applications. In recent years, computer-assisted proofs (CAPs) have emerged as a powerful tool to rigorously establish the existence of solutions to PDEs. The objective of CAPs is to combine numerical methods with rigorous mathematical analysis to provide precise, verified solutions, thereby bridging the gap between computational results and analytical proofs. By analyzing the solutions in an appropriately chosen space, CAPs offer a means to rigorously verify the existence of solutions expressed as expansions within a given basis. The challenge lies not only in obtaining a good numerical approximation but also in ensuring that the solution is  valid in a rigorous sense, respecting the underlying geometry of the problem and the specific properties of the PDE.

The field of CAPs in PDEs may be seen as part of a larger global effort to construct rigorous proofs for nonlinear dynamical systems. Early pioneering works, such as those on the Feigenbaum conjectures \cite{MR648529} and the existence of chaos and global attractors in the Lorenz equations \cite{MR1276767, MR1870856, MR1701385}, as well as more recent proofs of Jones' and Wright's conjectures in delay equations \cite{MR3912700, MR3779642}, chaos in the 1D Kuramoto-Sivashinsky PDE \cite{MR4113209}, and instability proofs in Poiseuille flow \cite{MR2492179}, have laid the groundwork for this field. Additionally, the study of bifurcating solutions for 3D Rayleigh-Bénard problems \cite{MR2470145}, equilibria in the 3D Navier-Stokes (NS) equations \cite{MR4372115}, solutions of NS on unbounded strips with obstacles \cite{Wunderlich2022_1000150609}, 3D gyroid patterns in materials \cite{MR3904424}, periodic orbits in NS \cite{MR4235780}, blowup in Euler equations on the cylinder \cite{CheHou22}, and imploding solutions for 3D compressible fluids \cite{BucCaoGom22}, further underscore the growing importance of CAPs in analyzing nonlinear systems. For more details, we refer the interested reader to the book \cite{MR3971222} and the survey papers \cite{MR1420838, MR3444942, MR3990999}. It is worth mentioning that a fundamental tool in these CAPs is the use of interval arithmetic \cite{Moore}, which efficiently ensures that all floating-point errors are controlled in the computations, thus guaranteeing the reliability of the results.

A crucial aspect of solving PDEs is the choice of basis functions in which the solution is expanded. The choice of basis function depends heavily on the geometry of the domain, the boundary conditions imposed, and the differential operators involved. For simple geometries, such as periodic domains or intervals, standard bases such as Fourier series or Chebyshev polynomials are often sufficient. For instance, on a periodic domain, Fourier basis functions are optimal, due to their simple form under differentiation and their good approximation properties. Similarly, on a one-dimensional compact interval, Chebyshev polynomials (essentially Fourier cosine series, as discussed in \cite{boyd2001chebyshev}) are widely used, particularly when boundary conditions are not periodic. However, when dealing with more complex domains --- such as disks, spheres, or other non-Euclidean geometries --- the selection of an appropriate basis becomes less straightforward. The Fourier basis, though highly effective for periodic problems, is not universally applicable in these cases. In such settings, the geometry of the problem suggests the use of specialized basis functions, which may include non-polynomial options like Bessel functions or Hermite functions, or polynomial options such as Legendre polynomials or generalized Jacobi polynomials. Each of these bases has its own advantages and drawbacks, making the choice of basis a key consideration in the numerical resolution of the PDE.

In the field of CAPs for PDEs, two distinct approaches have traditionally been employed (we refer the interested reader to the recent survey paper \cite{MR3990999} for a comprehensive discussion of the topic). First, finite element methods have been widely used to analyze PDEs defined on complex geometries \cite{MR1231898,MR2161437,MR3061473,Takayasu2013,MR3971222,Wunderlich2022_1000150609}. This approach is highly versatile, as it allows for the consideration of more general domains. However, it comes with the tradeoff of lacking strong approximation properties, which often limits the scope of CAP results to relatively simple dynamical objects, such as steady states and their stability. Second, there exists a long tradition of studying PDEs on toroidal or rectangular domains using Fourier series \cite{MR1838755,MR2718657,MR2728184,MR3904424,MR4254059}. The main advantage of this approach lies in the excellent approximation properties of spectral methods, including spectral convergence. This enables the analysis of more complex dynamical systems, such as periodic orbits \cite{MR3662023,MR4235780}, connecting orbits \cite{MR3773757} and chaos \cite{MR4113209}. However, the major limitation is that the Fourier approach is typically limited  to  restricted geometries, where the domain is rectangular or periodic. 

Classical topics in PDEs explore complex dynamical behaviors across a broad range of geometries, such as spiral waves in excitable media \cite{sandstede1999bifurcations} and the axisymmetric Navier-Stokes equations \cite{lopez1994bifurcation, hou2021nearly}. However, a gap remains in the literature on CAPs with respect to the development of spectral methods for such geometries. One key challenge in cylindrical and spherical domains is the apparent singularity introduced by polar coordinates. Previous work in the CAPs literature has addressed this difficulty by utilizing Taylor series expansions \cite{van2022rotation, van2018computer}. Recent efforts have sought to bridge this gap by establishing a framework for developing CAPs based on global spectral bases \cite{arioli2019non, maxime_hugo, arioli2024periodic}. In \cite{arioli2019non}, which focused on solutions on a disk, the authors noted that expanding the radial variable as a Chebyshev series proved effective as a pseudospectral method, but it was only by employing Zernike series that they successfully produced a CAP. In their subsequent work \cite{arioli2024periodic}, they used spherical harmonics as a basis to construct CAPs for periodic and quasi-periodic solutions on the sphere. Their results demonstrate that by combining spectral methods with rigorous error bounds, it is possible to establish the existence of solutions to elliptic and parabolic PDEs in geometries that were previously not considered in earlier studies.

Given the inherently nonlinear nature of many PDEs encountered in real-world applications, it is crucial to handle the nonlinear terms with care within the framework of CAPs. To address this challenge, the present work establishes a framework for rigorously evaluating nonlinearities within a chosen basis through the introduction of a validated {\em Matrix Multiplication Transform} (MMT) approach \cite{boyd2001chebyshev}. Similar to the discrete Fourier transform (DFT), the MMT approach facilitates the transformation of a function’s representation between its coefficients in a given basis (where derivative evaluations are straightforward) and its values at carefully selected grid points (where nonlinearities are more easily evaluated). The MMT provides a precise and computationally efficient method for evaluating nonlinear terms, which is typically a computationally intensive task requiring sophisticated error analysis. In the context of PDEs posed on higher-dimensional domains, where solutions are represented using a tensor product of bases, the coefficients-to-grid transformation can be performed through a sequence of one-dimensional transforms, leading to significant computational savings. Therefore, in this paper, we focus on developing the MMT approach specifically in the context of one-dimensional polynomials. 

To describe the MMT approach in  detail,  fix a one-dimensional polynomial basis $ \{ p_n \}_{n=0}^\infty$  on the interval $[a,b]$ such that the degree of $p_n$ is $n$, and consider a   polynomial $f:[a,b] \to \R$  
\begin{equation} \label{eq:Interpolation}
f(x) = \sum_{n=0}^N a_n p_n(x),
\end{equation}
given as a linear combination of basis polynomial functions $p_n(x)$. We aim to rigorously evaluate a nonlinear transformation $\mathcal{G}(f(x))$, where $\mathcal{G}$ is a polynomial nonlinear function, such as $\mathcal{G}(f(x)) = f(x)^2$. 
That is, we wish to derive an exact formula for  the expansion:
\begin{align}\label{eq:G(f)}
\mathcal{G}(f(x)) = \sum_{n=0}^{N'} c_n p_n(x),
\end{align}
%
where $N'$ is sufficiently  large (e.g.\ if $f$ is degree $N$ and $\cG $ is degree $d$, then $ N' = d N$). 

The  MMT approach provides a structured and efficient way to handle such nonlinear evaluations rigorously. 
Summarized in Figure \ref{fig:CommDiagram}, our approach involves: (i)  evaluating $f$ on a carefully selected grid; (ii) applying the nonlinearity to obtain the values of  $\mathcal{G}(f)$ on the grid; and (iii)   applying  an inverse transform to yield the coefficients of $\mathcal{G}(f)$ expanded in the basis $\{p_n\}_{n=0}^{N'}$. 
We denote the \emph{Matrix Multiplication Transform} (MMT) as the mapping from  coefficient-space to grid-space, and the  \emph{inverse Matrix Multiplication Transform} (iMMT) as the mapping  from  grid-space  to  coefficient-space.

\begin{definition}[\bf Matrix Multiplication Transform] \label{def:MMT}

Let $\bm{a}=(a_0,\dots,a_N)^T$ be the vector of coefficients of the finite sum of polynomials in \eqref{eq:Interpolation}, and let $\bm{f}\bydef (f(x_0), \dots, f(x_N))^T$ be the vector of grid values of \eqref{eq:Interpolation} at  distinct nodes $\{x_j\}_{j=0}^N$. 
\begin{itemize}
    \item Define $M:\bm{a}\mapsto \bm{f}$, which maps the coefficients $\bm{a}$ to the grid values $\bm{f}$, to be the \emph{MMT matrix}. 
    \item  Define $M^{-1}:\bm{f}\mapsto \bm{a}$, which maps the grid values to coefficients, to be  the \emph{iMMT matrix}. 
\end{itemize}

\end{definition}

As the polynomials $\{p_j\}_{j=0}^N$ are linearly independent and  nodes $\{x_j\}_{j=0}^N$ are distinct, it follows that  $M$ is an invertible, linear map, and may  be represented as an $(N+1)\times(N+1)$ matrix. 
Moreover, each entry of the MMT matrix $M$ is universally defined as
\begin{align}\label{eq:MMT_matrix}
    (M)_{j,n} \bydef p_n(x_j),
\end{align}
for $0\le j,n\le N$, where $j$ represents the row index and $n$ represents the column index of $M$. In this manner, we are able to obtain the grid values via the matrix-vector product $\bm{f}= M \bm{a}$. 

An essential consideration for CAPs is the effect of aliasing. The MMT is only an invertible transformation between the space of polynomials of degree at most $N$, and their values on nodes $\{x_j\}_{j=0}^N$. However  if $f$ is degree $N$ and $\cG $ is degree $d$, then $\cG(f)$  will be a degree $ N' = d N$ polynomial. To dealias our results, we need to  evaluate the nonlinearity $\mathcal{G}$ at the polynomial $f$ on  a set of grid points $\{x_j\}_{j=0}^{N'}$. Hence, a preliminary step in the MMT approach is to pad the coefficients  $\{a_n\}_{n=0}^N$ with zeros, mapping it to  $\{a_n\}_{n=0}^{N'} \bydef \iota^{N'}(\{a_n\}_{n=0}^N)$, where the inclusion $\iota^{N'}: \R^{N+1} \to \R^{N+1} \times \R^{N'-N} $ is  the  zero section. Subsequently, the $(N'+1)\times (N'+1)$ MMT matrix  enables us to transform the coefficients $\{a_n\}_{n=0}^{N'}$ to grid values $\{f(x_j)\}_{n=0}^{N'}$. The nonlinearity is then evaluated at the grid points (i.e., $f(x_j) \mapsto \mathcal{G}(f(x_j))$), and finally, the values $\mathcal{G}(f(x_j))$ are transformed back via the  iMMT to obtain the coefficients $c_n$ (see  Figure \ref{fig:CommDiagram}).

\begin{figure}[t]
\[
\begin{tikzcd}
\text{Coefficients:}\arrow[d] & \{a_n\}_{n=0}^N\arrow[hookrightarrow,r,"\iota^{N'}"] & \{a_n\}_{n=0}^{N'} \arrow[r] \arrow[d, "\text{MMT}"] & \{c_n\}_{n=0}^{N'} \\
\text{Grids:}\arrow[u]& &\{f(x_j)\}_{j=0}^{N'} \arrow[r]           & \left\{\cG(f(x_j))\right\}_{j=0}^{N'}\arrow[u, "\text{iMMT}"] 
\end{tikzcd}
\]
\caption{Diagram of evaluating the polynomial nonlinearity $\mathcal{G}$ via the MMT approach.}
\label{fig:CommDiagram}
\end{figure}
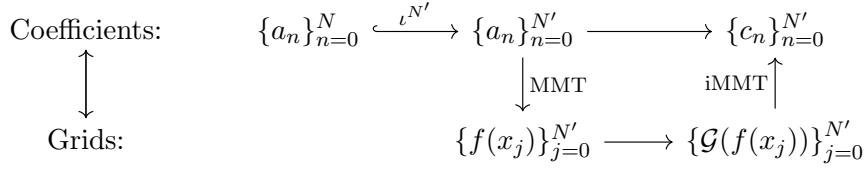

The selection of grid points $\{x_j\}_{j=0}^{N'}$ and polynomials $\{p_n(x)\}_{n=0}^{N'}$ is crucial for ensuring well-conditioned MMT matrices. For instance, if we choose monomials $p_n (x) = x^n$ and an equally spaced grid, the resulting matrix $M$ becomes a Vandermonde matrix, which is typically ill-conditioned, especially for large $N$. Therefore, making careful choices for these elements is essential. Classical Gaussian quadrature theory suggests that better conditioning can be achieved by selecting orthogonal polynomials $\{p_n(x)\}$ and placing the nodes $\{x_j\}_{j=0}^{N'}$ at the zeros of the corresponding polynomial of degree 
$N'+1$. However, even in this case, implementing the algorithms with validated numerics remains ill-conditioned and presents significant challenges for obtaining practical results \cite{storck1993verified}.

In Section~\ref{sec:orthogonal_polynomials_and_MMT}, we propose a general strategy for overcoming these difficulties. We provide explicit formulas for the entries of the iMMT matrix and introduce computable methods for efficiently calculating these entries. In particular, we provide in Section \ref{sec:validated_numerics_comparisons} a computer-assisted approach based on validated numerics to make such computations rigorous. We present a quantitative comparison of the different computable methods for computing the MMT and their relative propagation of error when using interval arithmetic. Such an analysis is essential in view of treating nonlinear PDE problems.

\begin{remark}[\bf Clebsch-Gordon vs MMT]\label{rem:CGvsMMT}
We note that when the polynomials $p_n$ in \eqref{eq:Interpolation} are orthogonal
, a conventional approach for evaluating $\mathcal{G}(f(x))$ employs Clebsch-Gordon (or linearization) coefficients \cite{biedenharn1984angular,arioli2019non}. These coefficients are defined through the relation $p_i(x) p_j(x) = \sum_{|i-j| \leq k \leq |i+j|} c_{i,j}^k p_k(x)$. However, this method generically incurs a computational complexity of $\cO(N^3)$ for each product evaluation. In contrast, our MMT approach reduces this complexity to $\cO(N^2)$, offering significant computational savings. While still less efficient than the Fast Fourier Transform (FFT), which operates with a complexity of $\cO(N \log N)$, the MMT approach remains well-suited for medium-sized problems \cite{boyd2011comparing,vasil2016tensor}. It is also worth emphasizing that in our complexity comparison
we do not account for the initial cost of precomputing the Clebsch-Gordon coefficients or the matrix coefficients required for the MMT method.
\end{remark}

The primary motivation for the validated MMT approach is to solve nonlinear PDEs using computer-assisted proofs. This work is particularly motivated by the study of the three-dimensional axisymmetric Navier--Stokes equations \cite{lopez1994bifurcation, lopez1998efficient, hou2008dynamic, hou2021nearly}. The axisymmetric assumption reduces the problem to a non-local two-dimensional PDE, dependent on the radial variable $r$ and the height variable $z$. However, applying current CAP techniques to these complex PDEs presents three main challenges. First, the PDEs involve a cylindrical Laplacian, requiring a specialized spectral basis for expansion in the radial variable \cite{arioli2019non}. Second, the equation contains a singular $1/r$ term within the nonlinearities. Lastly, the nonlinearities involve derivatives, further complicating the analysis. These challenges provide a strong motivation for employing the validated MMT approach, which offers a means to overcome these obstacles and provide rigorous solutions.

Our approach to tackling these difficulties involves using orthogonal polynomials defined on the disk, which necessitates the validated MMT approach derived in Section~\ref{sec:orthogonal_polynomials_and_MMT}. In Section~\ref{sec:Zernike}, we focus on the special case of Zernike polynomials, which are orthogonal polynomials on the unit disk. In particular, we review their fundamental properties and formulae for linear operators expressed in this basis, which solves the difficulties above in principle. Additionally, we detail how to compute products in the Zernike basis using the MMT, with the computations made rigorous via interval arithmetic, as detailed in Section~\ref{sec:validated_numerics_comparisons}.

To demonstrate the applicability of our approach,  we analyze in Section~\ref{sec:PDE} some toy problems on the unit disk $\bD \bydef \{z \in \mathbb{C}, ~ |z| \leq 1\}$  having quadratic nonlinearities.
In particular, given some $m \in \mathbb{N}_0=\{0,1,2,\dots,\}$, we study the following complex valued PDE:
\begin{align}\label{eq : zero finding original}
    \begin{cases}
        \triangle v(z) + \bar{z}^m v(z)^2 =0,~~ &\text{ for all } z \in \bD \\
        v(e^{i\theta}) =0,~~ &\text{ for all } \theta\in (0,2\pi).
    \end{cases}
\end{align}
Additionally we consider the following boundary value problem possessing the singular term $z^{-1}$.  
\begin{align}\label{eq : zero finding original 1/z}
    \begin{cases}
        \triangle v(z) + {z}^{-1} v(z)^2 =0,~~ &\text{ for all } z \in \bD\\
        v(e^{i\theta}) =0,~~ &\text{ for all } \theta\in (0,2\pi).
    \end{cases}
\end{align}
The singular inhomogeneous term $z^{-1}$ is meant to be a toy model version of similar terms which appear in the axisymmetric Navier--Stokes equations \cite{hou2008dynamic,majda2002vorticity}. 
In Theorem~\ref{thm:disk_solutions}, we prove the existence of solution to \eqref{eq : zero finding original} (for $m=0,1,2,20$) and \eqref{eq : zero finding original 1/z} (e.g. see Figure~\ref{fig:m_minus_one}).
\begin{figure}[H]\centering
	\begin{minipage}[b]{0.5\linewidth}
	\centering
        \includegraphics[width=\textwidth]{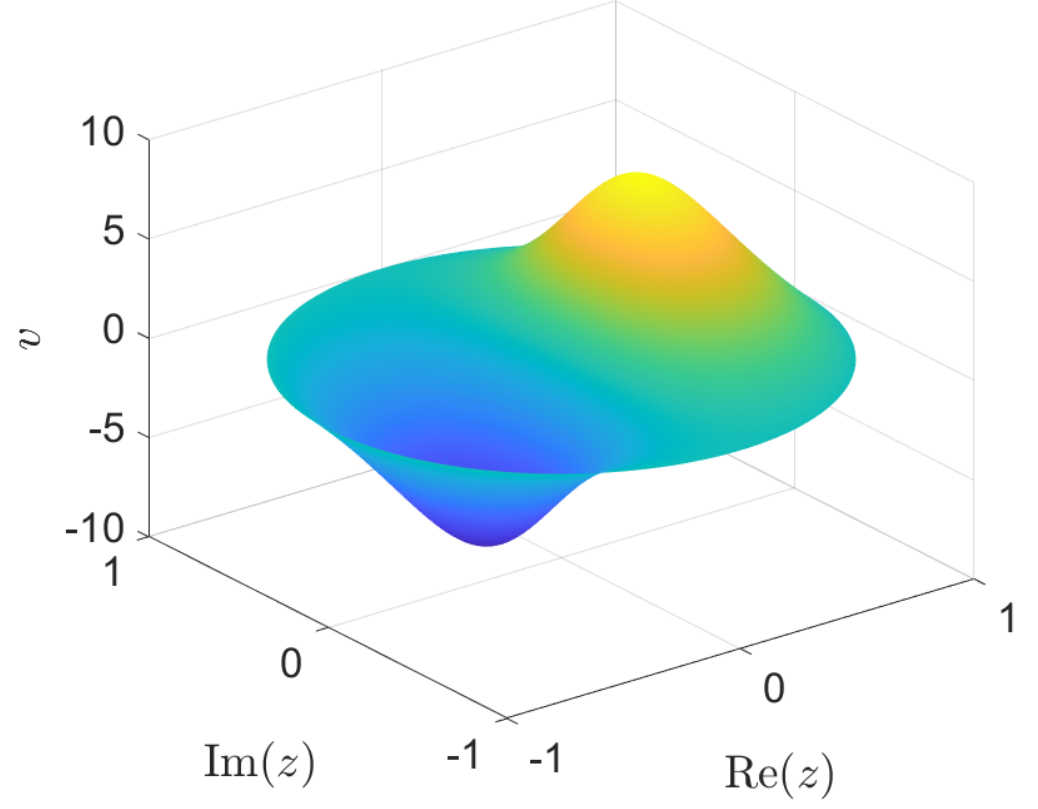}
	\end{minipage}
	\begin{minipage}[b]{0.47\linewidth}
	\centering
	\includegraphics[width=\textwidth]{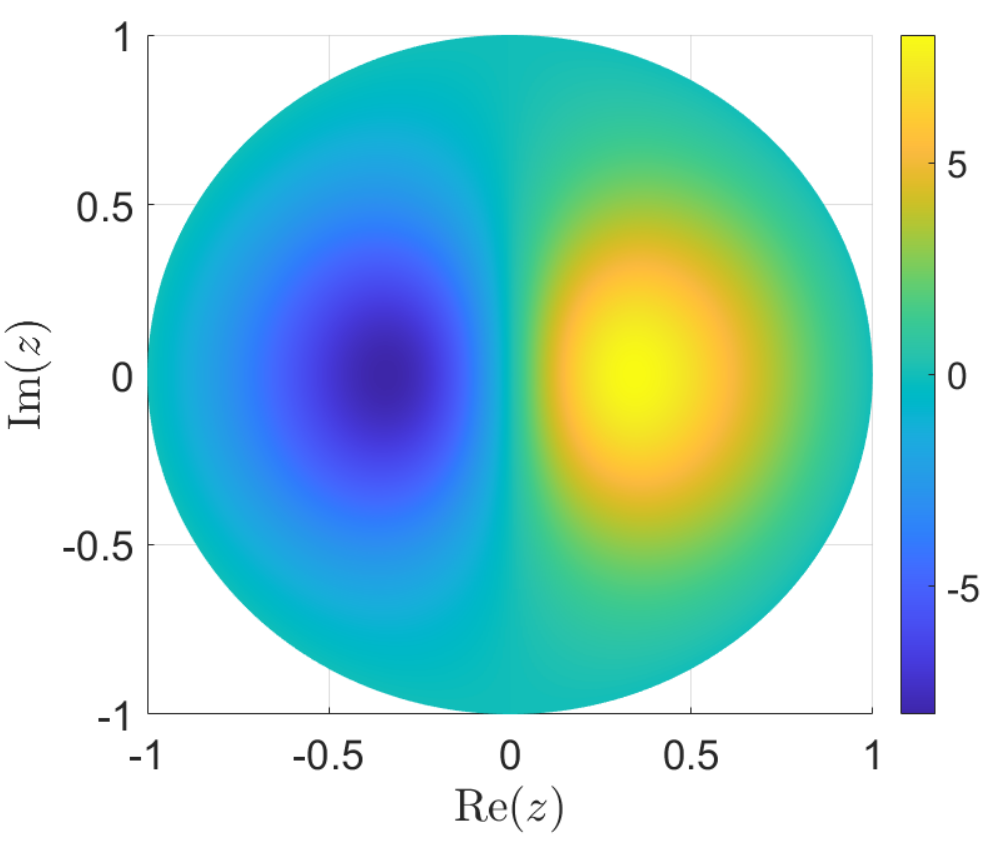}
	\end{minipage}
	\caption{Real part of a numerically computed approximate solution to \eqref{eq : zero finding original 1/z}.}\label{fig:m_minus_one}
\end{figure}

In future work we plan to extend this approach to the axisymmetric Navier--Stokes equations. This extension will require addressing the challenges posed by derivatives in the nonlinear terms arising from cylindrical geometry. By leveraging the validated MMT framework, we aim to rigorously analyze these equations and provide a robust foundation for solving complex fluid dynamics problems with axisymmetric assumptions.

\section{Orthogonal polynomials and the MMT approach}
\label{sec:orthogonal_polynomials_and_MMT}

As discussed earlier in Section~\ref{sec:Introduction}, the choice of grid points, $\{x_j\}_{j=0}^N$, and polynomials, $\{p_n(x)\}_{n=0}^N$, is crucial for preventing the formation of ill-conditioned MMT matrices. Therefore, making informed selections for these components is essential. A common method involves choosing polynomials $p_0, \dots, p_N$ that are orthogonal with respect to an appropriately defined inner product, with the grid points $x_0, \dots, x_N$ corresponding to the zeros of the polynomial $p_{N+1}$. In the present section, we outline a general framework for this approach and provide explicit formulas for the entries of the iMMT matrix, along with practical methods for computing these entries. To establish these concepts, we review key elements of orthogonal polynomial theory, Gaussian quadrature and numerical analysis.

\subsection{Background on orthogonal polynomials and Gaussian quadrature}

Let $\{p_n\}_{n=0}^\infty$ be a sequence of real-valued polynomial such that the degree of $p_n(x)$ is exactly $n$, that is $p_n(x) = k_nx^n + \cO(x^{n-1})$, where $k_n\neq 0$ is called the leading coefficient of $p_n$. Given an open interval $(a,b)$ ($-\infty\le a<b\le \infty$), let $\omega(x)$ be a generic weight function such that
\begin{equation}
    \omega(x)>0,~\text{for all } x\in (a,b)\quad\mbox{and}\quad\omega\in L^1(a,b).
\end{equation}
We say that two functions $f$ and $g$ are orthogonal with respect to $\omega$  if their inner product $(f,g)_\omega$ satisfies
\begin{equation}
    (f,g)_\omega\bydef\int_a^b f(x)g(x)\omega(x) dx=0.
\end{equation}
Moreover, a sequence of polynomials $\{p_n\}_{n=0}^\infty$ is said to be orthogonal with respect to $\omega$ if
\begin{equation} \label{eq:def_orthogonal_polynomials}
(p_n,p_m)_\omega = \int_a^b p_n(x)p_m(x)\omega(x)dx = \begin{cases}
0, & m\neq n,\\ W_n, & m=n,
\end{cases}
\end{equation}
where 
\begin{equation} \label{eq:W_n_scalar}
W_n \bydef \|p_n\|_\omega^2 = \int_a^b p_n(x)^2\omega(x)dx
\end{equation}
is a nonzero constant. If $W_n=1$ for all $n$, then the polynomials are said to be \emph{orthonormal}. Otherwise, these are called \emph{orthogonal polynomials} (with the weight function $\omega$).

A standard result (e.g. see Corollary 3.1 in \cite{shen2011spectral}) states that for a given weight function $\omega(x)$, there exists a unique sequence of orthogonal polynomials $\{p_n\}_{n=0}^\infty$ with leading coefficients $k_n$, given by the three-term recurrence relation
\begin{equation}\label{eq:recurrence_formula}
p_{n+1} = (\alpha_{n}x-\beta_{n}) p_n - \gamma_{n} p_{n-1},\quad n\ge 0
\end{equation}
with $p_{-1}=0$, $p_0=k_0$, and with
\begin{equation}\label{eq:abc}
\alpha_n \bydef \frac{k_{n+1}}{k_n},\qquad \beta_n \bydef  \frac{k_{n+1}}{k_n} \frac{\left(x p_n, p_n\right)_\omega}{\|p_n\|_\omega^2} \qquad \text{and} \qquad \gamma_n \bydef  \frac{k_{n-1} k_{n+1}}{k_n^2} \frac{\|p_n\|_\omega^2}{\|p_{n-1}\|_\omega^2}.
\end{equation}

Jacobi polynomials form a significant and extensive class of orthogonal polynomials, playing a central role in the applications we propose in this paper. We now proceed to introduce them.

\begin{definition}[\bf Jacobi polynomials] \label{def:jacobi_polynomials}
Given numbers $k,m>-1$ and the weight $\omega(x) \bydef (1-x)^k (1+x)^m$ (sometimes we simply refer to the weight as the pair $(k,m)$), the Jacobi polynomials $P_n^{k,m}(x)$ defined in $(-1,1)$ are orthogonal with respect to $\omega$, and they satisfy the recurrence formula $P_{n+1}^{k,m} (x) = (\alpha_n x -\beta_n ) P_n^{k,m}(x) - \gamma_n P_{n-1}^{k,m} (x)$, for $n\ge 0$, with $P^{k,m}_{-1}(x) = 0$ and $P^{k,m}_0(x) = 1$, where we have 
\tiny{
\[
\alpha_n = \frac{(2n+k+m+1)(2 n +k+m+2)}{2(n+1)(n+k+m+1)}, \quad 
\beta_n = \frac{(m^2 -k^2)(2n+k +m+1)}{2(n+1)(n+k+m+1)(2n+k+m)}, \quad
\gamma_n = \frac{(n+k)(n+m)(2n+k+m+2)}{(n+1)(n+k+m+1)(2n+k+m)}.
\]
}
\end{definition}

When the polynomials $p_n$ are orthogonal as in \eqref{eq:def_orthogonal_polynomials}, the coefficients $a_n$ in \eqref{eq:Interpolation} can be obtained by taking the inner product $(\cdot,p_n)$ on both sides of \eqref{eq:Interpolation}. This yields the expression 
\[
a_n = \frac{(f,p_n)_\omega}{W_n}, \quad n=0,\dots,N,
\]
which can be challenging to evaluate analytically for functions $f$ with intricate forms. However, if $f$ is a polynomial, Gaussian quadrature provides an efficient method for computing the inner product explicitly. To formalize this approach, let $ \{x_j\}_{j=0}^{N}$ be the zeros of the orthogonal polynomial $p_{N+1}$ with respect to the weight function $\omega$ and let $\{ \omega_j \}_{j=0}^N$ be the weights defined by 
\begin{equation} \label{eq:omega_j}
\omega_j\bydef \int_a^bh_j(x)\omega(x)dx,
\end{equation}
where the $h_j(x)$ are the standard Lagrange basis polynomials. Then (e.g. see Theorem\,3.5 in \cite{shen2011spectral} or Theorem\,14 of Sec.\,4.3 in \cite{boyd2001chebyshev}), the quadrature formula
\begin{equation} \label{eq:GaussianQuadrature}
\int_{a}^b f(x) \omega(x) dx = \sum_{0 \leq j \leq N} f(x_j) \omega_j
\end{equation}
holds for any $f$ being a polynomial of degree at most $2N+1$.

The following result (e.g. see Theorem\,18 of Sec.\,4.4 in \cite{boyd2001chebyshev}), which directly follows from \eqref{eq:GaussianQuadrature}, will enable us to derive an explicit formula for the entries of the iMMT matrix when $f$ is a polynomial.

\begin{theorem}[\bf iMMT matrix via Gaussian quadrature] \label{thm:Interpolation}
Let $\{p_n(x)\}_{n=0}^\infty$ be a sequence of orthogonal polynomials with respect to a weight function $w(x)$ on the interval $[a, b]$ and let $\{x_j\}_{j=0}^N$ be the $N+1$ distinct roots of the orthogonal polynomial $p_{N+1}(x)$. Then, any polynomial $f(x)$ of degree at most $N$ can be uniquely represented in the orthogonal polynomial basis as in \eqref{eq:Interpolation}, where the coefficients $a_n$ are given by
\begin{equation} \label{eq:an_coeffs}
a_n = \frac{(f,p_n)_\omega}{W_n} = \frac{1}{W_n} \int_a^b f(x) p_n(x) \omega(x) dx 
= \sum_{j=0}^N \frac{\omega_j p_n(x_j)}{W_n} f(x_j) ,\quad n=0,\dots,N.
\end{equation}
Hence, each entry of the iMMT matrix $M^{-1}$ is defined as
\begin{equation} \label{eq:iMMT_matrix}
(M^{-1})_{n,j} \bydef \frac{ \omega_j	p_{n}(x_j) }{W_n}, 
\end{equation}
where $n$ represents the row index and $j$ represents the column index of $B$.
\end{theorem}

In summary, to compute the entries of the iMMT matrix given in \eqref{eq:iMMT_matrix} and the MMT matrix defined in \eqref{eq:MMT_matrix}, it is essential to determine several key components, specifically:
\begin{equation} \label{eq:quantities_to_compute}
\begin{aligned}
& \bullet \{x_j\}_{j=0}^{N}: \text{the nodes (roots of } p_{N+1}) \\
& \bullet \{\omega_j \}_{j=0}^N: \text{the weights in the Gaussian quadrature in } \eqref{eq:omega_j} \\
& \bullet p_n(x_j): \text{the values of the orthogonal polynomials at the nodes via } \eqref{eq:recurrence_formula} \\
& \bullet W_n: \text{the scaling factors in } \eqref{eq:W_n_scalar}.
\end{aligned}
\end{equation}


We explain in the next section details how to perform each these computations rigorously and provide some explicit formulas in the context of Jacobi polynomials (e.g. see Definition~\ref{def:jacobi_polynomials}). 

\subsection{Computing the entries of the MMT and iMMT matrices}

To outline a framework for rigorously computing the entries of the MMT matrix $M$ defined in \eqref{eq:MMT_matrix} and the iMMT matrix $M^{-1}$ defined in \eqref{eq:iMMT_matrix}, it is essential to rigorously compute all the quantities specified in \eqref{eq:quantities_to_compute}. We proceed as follows: in Section~\ref{sec:compute_x_j}, we explain how to obtain the nodes $x_j$ by calculating the eigenvalues of a specific symmetric tridiagonal matrix $\cA_{N+1}$, as defined in \eqref{eq:matrix_A_Nplus1}. In Section~\ref{sec:compute_omega_j}, we show how to determine the weights $\omega_j$ using the eigenvectors of $\cA_{N+1}$. Section~\ref{sec:computing_p_n(x_j)}, introduces two methods for evaluating the orthogonal polynomials at the grid points $p_n(x_j)$: one utilizing the recurrence relation \eqref{eq:recurrence_formula}, and the other involving the solution of a linear system. In Section~\ref{sec:W_n}, we provide a brief discussion on deriving the formula for the scaling factor $W_n$ and present the explicit expression within the framework of Jacobi polynomials. It is important to note that these computations are inherently susceptible to rounding errors. Given that our goal is to produce computer-assisted proofs of the existence of solutions to PDEs, we must establish rigorous and efficient methods for controlling all errors in evaluating the quantities listed in \eqref{eq:quantities_to_compute}, which are crucial for accurately constructing the matrices $M$ and $M^{-1}$. This subject is thoroughly discussed in Section~\ref{sec:validated_numerics_comparisons}.


\subsubsection{Computing the nodes \boldmath$x_j$\unboldmath} \label{sec:compute_x_j}

We start by describing how to obtain the nodes $\{x_j\}_{j=0}^{N}$, which, it is important to recall, are the zeros of the orthogonal polynomial $p_{N+1}$. These nodes are crucial because they are also used to generate computational grids for spectral methods. While methods like Newton's algorithm could be employed to find the zeros of $p_{N+1}$, we present the more efficient \emph{Golub-Welsch algorithm} \cite{Golub-Welsch}, which is specifically designed to compute the zeros of orthogonal polynomials accurately.

Recalling the coefficients $\alpha_n$, $\beta_n$ and $\gamma_n$ be defined as in \eqref{eq:abc}, it is known (e.g. see Theorem\,3.4 in \cite{shen2011spectral}) that the zeros $\{x_j\}_{j=0}^N$ of the orthogonal polynomial $p_{N+1}$ are the eigenvalues of the following symmetric tridiagonal matrix
\begin{equation} \label{eq:matrix_A_Nplus1}
\cA_{N+1} = 
\begin{pmatrix}
\mu_0 & \eta_1 &&& \\
\eta_1 &\mu_1  &\eta_2 &&\\	
&\ddots &\ddots & \ddots & \\
&&\eta_{N-1} & \mu_{N-1}&\eta_N \\			
&&&\eta_N &\mu_N
\end{pmatrix}
\end{equation}
where $\mu_j = \frac{\beta_j}{\alpha_j}$ for $j \geq 0$ and $\eta_j= \frac{1}{\alpha_{j-1} } \sqrt{
\frac{\alpha_{j-1} \gamma_j}{\alpha_{j}}}$ for $j \geq 1$. The computation of an eigenvalue-eigenvector pair $(\lambda,u)$ of $\cA_{N+1}$ such that $\|u\|^2=1$ can be reformulated as looking for a solution of the 
\begin{equation} \label{eq:eigenpair_F=0}
F^{\rm eig}(\lambda,u) \bydef \begin{pmatrix}
(u,u)_2-1\\
\cA_{N+1} u - \lambda u
\end{pmatrix}.
\end{equation}
In Section~\ref{sec:validated_numerics_comparisons}, we present a Newton-Kantorovich theorem (see Theorem~\ref{thm::RadiiNonLin}) that, when combined with interval arithmetic, can be used to rigorously determine the locations of all eigenvalue-eigenvector pairs of $\cA_{N+1}$ by solving problem $\eqref{eq:eigenpair_F=0}$ $N+1$ times.

\subsubsection{Computing the weights \boldmath$\omega_j$\unboldmath} \label{sec:compute_omega_j}

With a strategy in place for computing the nodes $x_j$, the next step is to calculate the weights $\omega_j$ defined in \eqref{eq:GaussianQuadrature}. Let $Q(x_j) =( Q_0( x_j), \dots, Q_N(x_j))^T \in \R^{N+1}$ be an eigenvector of $\cA_{N+1}$ associated with the eigenvalue $x_j$ such that $\|Q(x_j)\|^2 = \left( Q(x_j),Q(x_j) \right)_2 = 1$. According to Theorem 3.6 in \cite{shen2011spectral}, the weights $\{ \omega_j\}_{j=0}^{N}$ can be obtained from the first component of the eigenvector $Q(x_j)$ using the formula
\begin{equation} \label{eq:omega_j_general}
\omega_j = \left[ Q_0(x_j)\right]^2 \int_a^b \omega(x) dx.
\end{equation}

\begin{remark}[\bf Computing the weights for Jacobi polynomials]\label{rem:omega_j_jacobi}
In case we are working with Jacobi polynomials, which are orthogonal with respect to the weight $\omega(x)=(1-x)^k(1+x)^m$ with $k,m \in \mathbb{N}_0$, Equation (3.144) from \cite{shen2011spectral} yields the explicit expression 
\begin{equation} \label{eq:omega_j_Jacobi}
\int_a^b \omega(x) dx=\int_{-1}^1 (1-x)^k(1+x)^m dx = \frac{2^{k+m+1} k! m!}{(k+m+1)!}.
\end{equation}
\end{remark}
For Jacobi polynomials, the only remaining task to compute $\omega_j$ in \eqref{eq:omega_j_general} is to compute the value $Q_0(x_j)$, which is obtained by computing the eigenvectors of the matrix $\cA_{N+1}$ defined in \eqref{eq:matrix_A_Nplus1}.


\subsubsection{Evaluating \boldmath$p_n(x_j)$\unboldmath: the orthogonal polynomials at the grid points} \label{sec:computing_p_n(x_j)} 

In this section, we present three distinct methods for evaluating $p_n(x_j)$ for $j, n = 0, \dots, N$, which correspond to the values of orthogonal polynomials at the grid points. These methods include: the Forsythe algorithm, which utilizes the recurrence relation; a linear system approach, solved using a Newton-Kantorovich method with an approximate inverse; and a technique based on the eigenvectors of the matrix $\cA_{N+1}$. \\

\noindent {\bf Method 1 (Forsythe algorithm).} 
Recalling the recurrence formula given in \eqref{eq:recurrence_formula}, a first approach to computing the values $p_n(x_j)$ is to use the previously obtained nodes $x_j$, by setting $p_{-1}=0$ and $p_0=k_0$ and then directly evaluating
\[
p_{n+1}(x_j) = (\alpha_{n}x_j-\beta_{n}) p_n(x_j) - \gamma_{n} p_{n-1}(x_j),\quad n\ge 0.
\]
This method, called the {\em Forsythe algorithm} \cite{MR0092208}, is commonly used in practice. However, it is well known that numerical stability can be a concern when computing recurrence relations, particularly when using interval arithmetic, leading to potential accuracy issues. To mitigate this, two approaches can be employed: using high-precision arithmetic or solving a linear system (see \eqref{eq:linear_system_definition} below) derived from the recurrence relation (referred to as the \emph{linear system approach}). In Section~\ref{sec:validated_numerics_comparisons}, we implement Methods 1 and 2 alongside interval arithmetic and compare their performance in Section~\ref{sec:ExampleValidated Numerics}. Before doing so, let us provide further details on our implementation of the linear system approach. \\

\noindent {\bf Method 2 (Linear System Approach).} For a fixed $x\in(a,b)$, rewriting \eqref{eq:recurrence_formula} into a matrix form 
\begin{equation} \label{eq:linear_system_definition}
    \bm{A}(x)\bm{p}(x) = \bm{e}_{1},
\end{equation}
where $\bm{A}(x)\in \R^{(N+1)\times (N+1)}$ and $\bm{p}(x)$, $\bm{e}_{1}\in\R^{N+1}$ are defined by
\begin{equation}\label{eq:matrix_F}
\bm{A}(x) \bydef \begin{pmatrix}
1 &  &  & &\\
\vartheta_0(x) & 1 & \ddots & \ddots & \\
\gamma_1 & \vartheta_1(x)  & 1 & \ddots & \\
&  &  & \ddots & \\
&  & \gamma_{N-1} & \vartheta_{N-1}(x) & 1
\end{pmatrix},\quad \vartheta_n(x) \bydef -(\alpha_n x - \beta_n),
\end{equation}
and
\begin{equation}
\bm{p}(x)\bydef \left(p_0(x),p_1(x),\dots,p_N(x)\right)^T,
\quad \bm{e}_{1}\bydef (1,\dots,0,0)^T,
\end{equation}
respectively. Then, setting $x=x_j$ for $j=0,\dots,N$ in \eqref{eq:matrix_F}, the solution of the linear system $\bm{p}(x_j) = \bm{A}(x_j)^{-1}\bm{e}_{1}$ gives the $j$-th column of the MMT matrix $M$, that is the value of the orthogonal polynomials at the grid points $p_n(x_j)$. To obtain a rigorous enclosure of the solution to $\bm{A}(x_j) \bm{p}(x_j) = \bm{e}_{1}$, we can use the Newton-Kantorovich Theorem~\ref{thm::RadiiNonLin} in conjunction with interval arithmetic to rigorously determine the location of the solution to the system defined by
\begin{equation} \label{eq:linear_system_approach_F=0}
F^{x_j}(p) \bydef \bm{A}(x_j) p - \bm{e}_{1}.
\end{equation}
We refer to this way of evaluating the orthogonal polynomials at the grid points, that is $p_n(x_j)$, as the \emph{linear system approach}. Although this approach may initially seem excessive for solving a linear system, the Newton-Kantorovich Theorem requires only an approximate inverse, which can be computed without interval arithmetic, thereby minimizing the wrapping effect. This is further explained in Section~\ref{sec:validated_numerics_comparisons}.

The third approach (e.g. see \cite{Golub-Welsch}) we introduce for evaluating the orthogonal polynomials at the grid points utilizes the eigenvectors of $\cA_{N+1}$. \\

\noindent {\bf Method 3 (via the eigenvectors of \boldmath$\cA_{N+1}$\unboldmath).} The matrix $\cA_{N+1}$ is derived from the similarity transformation $\cA_{N+1} = DJD^{-1}$. Here, $J$ is the following non-symmetric tridiagonal matrix, called the \emph{Jacobi matrix}:
\begin{align}\label{eq:J_matrix}
    J=\left(\begin{array}{ccccc}
		\beta_0 / \alpha_0 & 1 / \alpha_0 & & & \\
		\gamma_1 / \alpha_1 & \beta_1/\alpha_1 & 1 / \alpha_1 & & \\
		 & \ddots & \ddots & \ddots & \\
		& & \gamma_{n-1} / \alpha_{n-1} & \beta_{n-1} / \alpha_{n-1} & 1 / \alpha_{n-1} \\
		& & & \gamma_n / \alpha_n & \beta_n / \alpha_n
	\end{array}\right),
 \end{align}
 where $\alpha_n$, $\beta_n$, and $\gamma_n$ are the coefficients in the recurrence formula given in \eqref{eq:abc} satisfying the three-term recurrence formula \eqref{eq:recurrence_formula}.
 The diagonal similarity transformation is performed via the matrix $D$ given by
 \begin{align}\label{eq:D_matrix}
     D=
	\begin{pmatrix}
		d_0 & & & \\
		& d_1 & & \\
		& & \ddots & \\
		& & & d_n
	\end{pmatrix},\quad
	d_{i+1}=\sqrt{\frac{\alpha_{i+1}}{\alpha_i \gamma_{i+1}}} d_i\quad(d_0=1).
\end{align}
Considering an eigenpair $(x_j,Q(x_j))$ satisfying $\cA_{N+1} Q(x_j) = x_j Q(x_j)$, then
from the similarity transformation $\cA_{N+1} = DJD^{-1}$ given in \eqref{eq:J_matrix} and \eqref{eq:D_matrix}, it follows that
\[
DJD^{-1}Q(x_j)=x_j {Q}(x_j) \iff JD^{-1}{Q}(x_j)=x_j D^{-1}{Q}(x_j).
\]
Here, the eigenvalues of $J$ are the same as that of $\cA_{N+1}$. Since the matrix $J$ directly comes from the recurrence formula \eqref{eq:recurrence_formula}, eigenvectors of $J$, that is $D^{-1}{Q}(x_j)$, are equal to the values of the orthogonal polynomial, i.e., $(p_n(x_j))_{n=0}^N$, except for the scaling freedom.
To fix the scale satisfying $p_0=k_0$, we have
\begin{equation} \label{eq:evaluating_the_orthogonal_polynomials_at_grid_points}
    p_n(x_j)=\frac{(D^{-1}{Q}(x_j))_n}{(D^{-1}Q(x_j))_0}k_0,
\end{equation}
where $(D^{-1}Q(x_j))_n$ denotes the $n$-th component of the vector $D^{-1}{Q}(x_j)$. Formula \eqref{eq:evaluating_the_orthogonal_polynomials_at_grid_points} is yet another method to evaluating the orthogonal polynomials at the grid points, hence providing a way to compute the entries of the MMT matrix $M$ defined in \eqref{eq:MMT_matrix}.


\subsubsection{Computing the scaling factors \boldmath$W_n$\unboldmath} \label{sec:W_n}

Deriving analytic expressions for the scaling factors $W_n$ defined in \eqref{eq:W_n_scalar} for different families of orthogonal polynomials is a well-established area of study with a rich history. Various tables, such as those found in \cite{MR1773820}, provide these expressions for many significant families, including Legendre, Chebyshev, Hermite, and Laguerre polynomials. For the Jacobi polynomials $P^{k,m}_n$ with $k$ and $m$ non negative integers, it is given by 
\begin{equation} \label{eq:JacobiScalingFactors}
    W_n = 
    W_{n}^{k,m} = \frac{2^{k+m+1}}{2n+k+m+1}\frac{ (n+k)! (n+m)!}{(n+k+m+1)!n!}.
\end{equation}
For non-integer values of $k,m>-1$, this formula extends to the use of Gamma functions. Additionally, for validated numerical computations, it is crucial to handle cancellations carefully when dividing factorials by other factorials --- e.g. $ \frac{ (n+k)! (n+m)!}{(n+k+m+1)!n!} = \frac{\prod_{i=1}^k (n+i)}{\prod_{i=1}^{k+1} (n+m+i)}$ --- to avoid a significant increase in round-off errors.

\section{Validated numerics and comparisons} \label{sec:validated_numerics_comparisons}
 
The discussion of Section~\ref{sec:orthogonal_polynomials_and_MMT} has been a review of standard orthogonal polynomials, Gaussian quadrature and numerical analysis concepts. While Gaussian quadrature is theoretically exact when using the correct nodes $x_j$ and weights $\omega_j$, this requires an accurate calculation of these nodes and weights, as well as their evaluation under the relevant polynomials. However, computers use floating point arithmetic, which operates with finite precision. In numerical quadrature, the primary challenge to achieving accuracy often stems from rounding errors and the limitations of finite precision arithmetic \cite{barrio2002rounding,hale2013fast,bogaert2014iteration}.

As numerical quadrature serves as the foundation for subsequent computations, any imprecision introduced here will propagate. At a fundamental level, if a node is computed with machine precision, $x_i = \bar{x}_i \pm 10^{-16}$, the error in its evaluation under a function $f$ will be at least $|f(x_i) - f(\bar{x}_i)| \approx |f'(x_i)| \times 10^{-16}$. For Legendre polynomials, their derivative at the endpoint is given by $P_n'(1) = \frac{n(n+1)}{2}$. Hence, for a truncation at $N = 100$, the error could be amplified by four orders of magnitude. A more accurate calculation of $f(x_i)$ can be achieved by computing $x_i$ with multiple precision. While precision will still be lost when evaluating $f(x_i)$, starting with higher precision ensures that the final result is sufficiently accurate. However, using multiple precision requires more memory and is slower compared to double precision. 

To synthesize the various types of errors encountered in numerical computations, we employ validated numerics \cite{MR2807595,MR2652784,MR1849323,MR1420838,MR3444942}. A fundamental aspect of validated numerics is interval arithmetic, which rigorously bounds rounding errors. This approach allows us to derive explicit \emph{a posteriori} error bounds for our results. Although limited work has been done in the area of Gaussian quadrature using validated numerics \cite{storck1993verified}, these methods enable us to set specific accuracy goals; for instance, we may aim for the entries of the MMT/iMMT matrices $M$ and $M^{-1}$ to achieve an accuracy of $10^{-16}$. While obtaining the nodes with high (and therefore slower) precision may be necessary, we can subsequently store the computed matrices $M$ and $M^{-1}$ as a double matrix.

Recall from Sections~\ref{sec:compute_x_j} and~\ref{sec:compute_omega_j} that computing the nodes $\{ x_j \}_{j=0}^N$ and weights $\{ \omega_j \}_{j=0}^N$ first requires determining all eigenpairs of the matrix $\cA_{N+1}$ defined in \eqref{eq:matrix_A_Nplus1}. We interpret each eigenpair as a zero of the map specified in \eqref{eq:eigenpair_F=0}. Furthermore, as described in Section~\ref{sec:computing_p_n(x_j)}, Method 2 for evaluating $ p_n(x_j)$ involves selecting a grid point $x_j$ and finding the zero of the map given in \eqref{eq:linear_system_approach_F=0}. To solve these two zero-finding problems, we apply a Newton-Kantorovich theorem that rigorously encloses the zeros, presented below in a general setting for maps between Banach spaces. This theorem offers rigorous \emph{a posteriori} estimates for locating zeros of a function. Moreover, we will use this theorem to establish the existence of solutions to elliptic semilinear PDEs on the disk, as detailed in Section~\ref{sec:PDE}.

\begin{theorem}[\bf Newton-Kantorovich Theorem] \label{thm::RadiiNonLin}
Let $X$ and $X'$ be Banach spaces and $F: X \to X'$ be a Fr\'echet differentiable mapping. Suppose ${x}_0 \in X$ and $A \in B(X',X)$. Moreover assume that $A$ is injective. Let $Y_0$ and $ Z_1$ be positive constants and $Z_2:(0,\infty) \to [0,\infty)$ be a non-negative function satisfying 
\begin{align*}
\| AF({x}_0) \|_X &\leq Y_0, \\
\| I - A DF({x}_0)\|_{B(X)} &\leq Z_1, \\
\|A[DF(c) - DF({x}_0)] \|_{B(X)} &\leq Z_2(r)r, \quad \text{ for all } \quad c \in \overline{B_r({x}_0)} \text{ and all } r>0.
\end{align*}
Define
\begin{align*}
p(r) = Z_2(r)r^2 - (1-Z_1)r + Y_0.
\end{align*}
If there exists $r_0>0$ such that $p(r_0) < 0 $, then there exists a unique $\tilde{x} \in \overline{B_{r_0}({x}_0)} $ satisfying $F(\tilde{x}) = 0$.
 \end{theorem} 

Typically, the vector $x_0$ is a numerical approximation, while the operator $A$ serves as an approximate inverse of the derivative $DF(x_0)$.

\subsection{Comparison: Evaluation of the Jacobi Polynomials}\label{sec:ExampleValidated Numerics}

We now turn to the evaluation of Jacobi polynomials as introduced in Definition~\ref{def:jacobi_polynomials}. Specifically, we consider the case $k=m=1$ and analyze the performance of the first two methods described in Section~\ref{sec:computing_p_n(x_j)}: the Forsythe algorithm and the linear system approach, which employs Theorem~\ref{thm::RadiiNonLin} to determine the zeros of the map \eqref{eq:linear_system_approach_F=0}. We compare these methods under a fixed working precision, focusing on both runtime performance (illustrated in Figure~\ref{fig:1_1}) and the error associated with the evaluation of the Jacobi polynomials (shown in Figure~\ref{fig:1_2}). For the runtime assessment, we varied the degree $N$ of the Jacobi polynomial and computed $P_N^{k,m}$ at a random point in $[-1,1]$.

As depicted in Figure~\ref{fig:1_1}, Forsythe algorithm significantly outperforms the linear system approach in terms of runtime, owing to its computational complexity of $\cO(N)$ compared to the $\cO(N^3)$ complexity of the linear system approach. However, as illustrated in Figure~\ref{fig:1_2}, the interval radii produced by Forsythe algorithm exhibit much greater instability. In contrast, while the interval radii for the linear system approach also increase with $N$, they are more controlled. This discrepancy arises from the inherent wrapping effect associated with recursion-based algorithms, for which interval arithmetic is particularly ill-suited. Additionally, for both methods, the output interval radii are insufficient to achieve $10^{-16}$ accuracy for the MMT matrix $M$ defined in \eqref{eq:MMT_matrix}. This limitation stems from the large magnitudes of polynomial values as $N$ increases. Therefore, employing higher precision in numerical computations may be necessary.

\begin{figure}[h!]\centering
	\begin{minipage}[b]{0.48\linewidth}
		\centering
		\includegraphics[width=\textwidth]{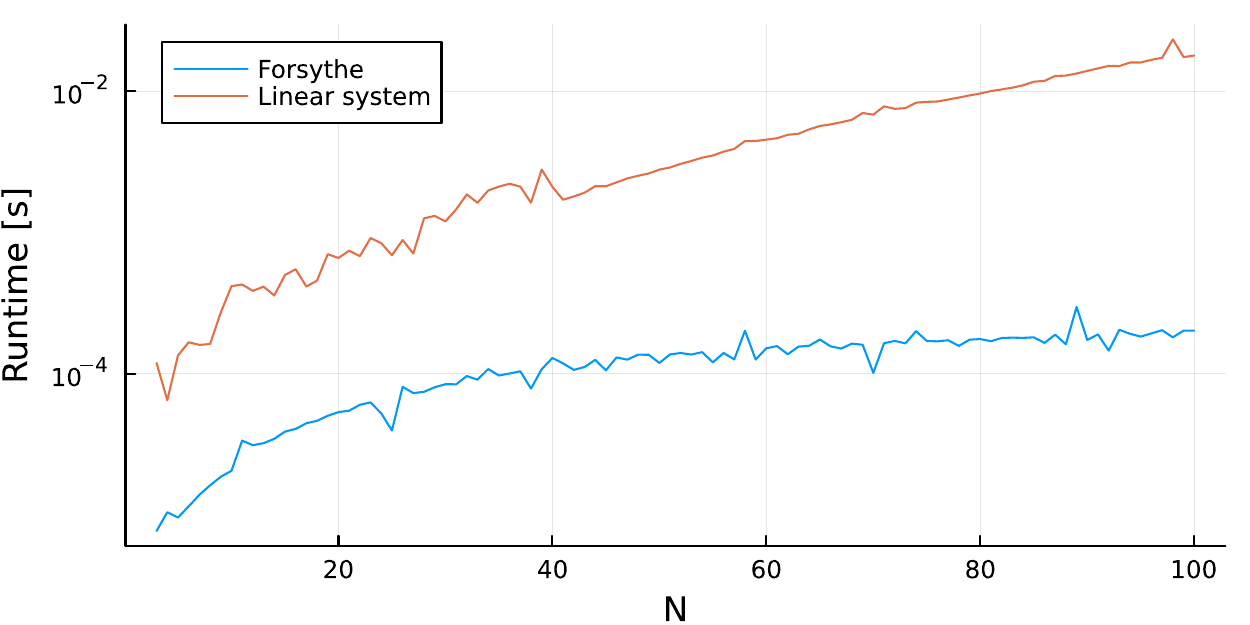}
		\subcaption{The plot illustrates the runtime (in seconds) as a function of the polynomial degree $N$ for evaluating $P^{1,1}_N(x)$ at a random point $x \in [-1,1]$ using interval arithmetic. It is evident that the Forsythe algorithm is significantly faster than the linear system approach.\\ \null} 
        \label{fig:1_1}
	\end{minipage}~
	\begin{minipage}[b]{0.48\linewidth}
		\centering
		\includegraphics[width=\textwidth]{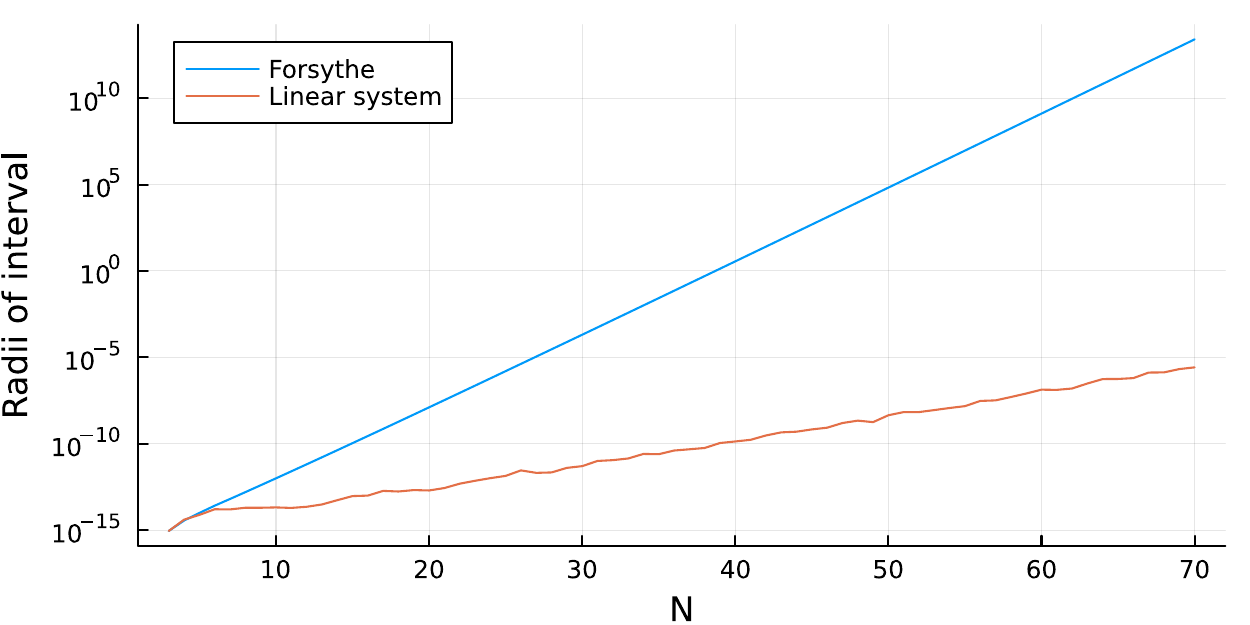}
		\subcaption{The relationship between the polynomial degree $N$ and the maximum radius of the output interval inclusions for evaluating $P^{1,1}_N$ at 50 random points within $[-1,1]$ is shown. We observe that the error bounds for the Forsythe algorithm increase rapidly, primarily due to the wrapping effect that arises when using interval arithmetic for the recursion formula.}\label{fig:1_2}
	\end{minipage}
	\caption{Comparison of evaluating $P_n^{1,1}$ with interval arithmetic via Forsythe algorithm and the linear system approach in terms of (a) running time and (b) resulting error in evaluations.}
    \label{fig:1}
\end{figure}

\begin{figure}[h!] \centering
	\begin{minipage}[b]{0.48\linewidth}
		\centering
		\includegraphics[width=\textwidth]{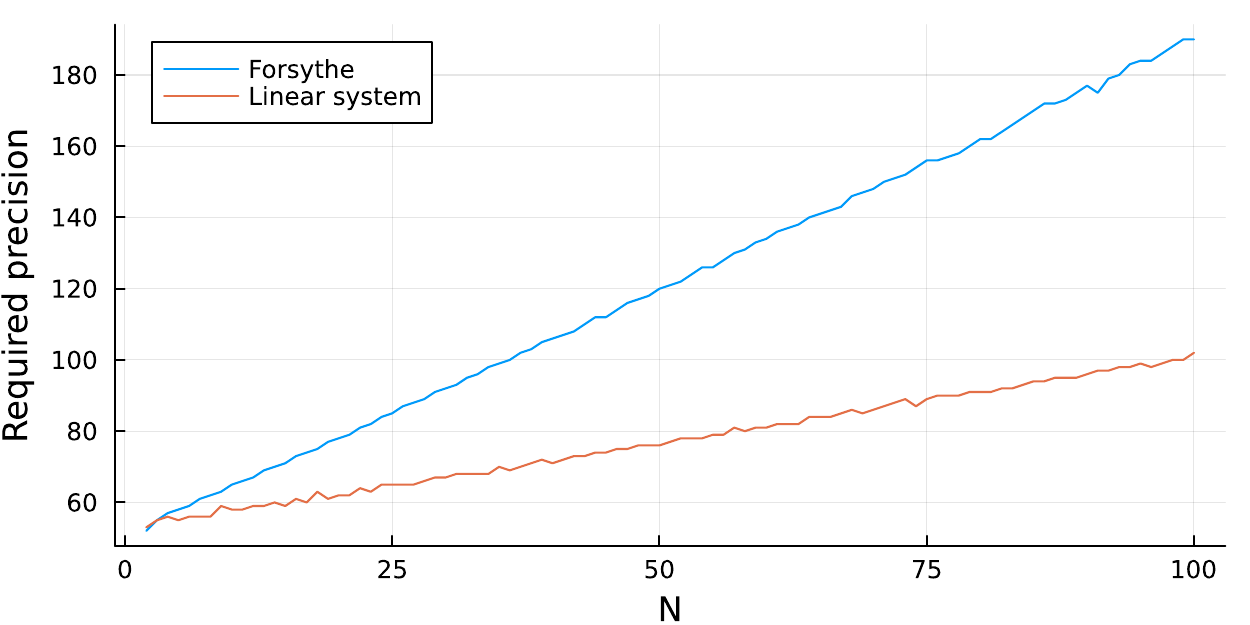}
		\subcaption{In the plot, we illustrate the precision of the arithmetic (as a function of $N$ required to ensure that the output radius of the interval inclusion of $P_N^{1,1}(x)$ is below the machine precision $\epsilon$ for 64-bit floating-point numbers. It is observed that both algorithms require more than $N$-precision (which corresponds to high-precision arithmetic when $N > 64$ to achieve the machine precision $\epsilon$.}
        \label{fig:2_1}
	\end{minipage}~
	\begin{minipage}[b]{0.48\linewidth}
		\centering
		\includegraphics[width=\textwidth]{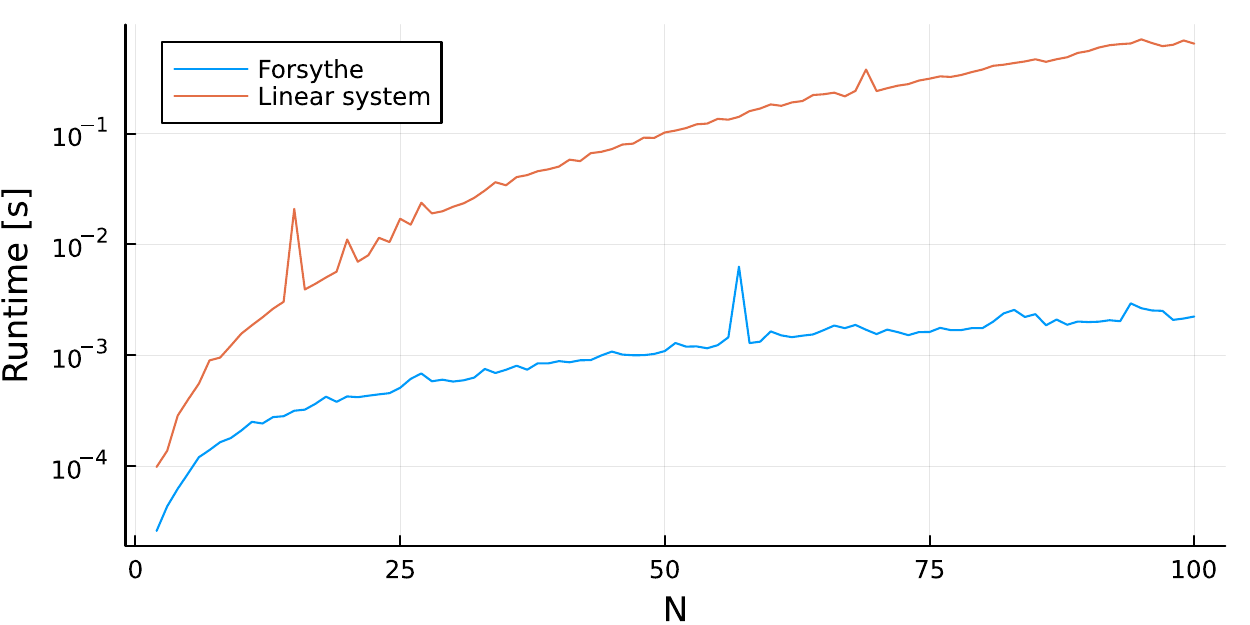}
		\subcaption{The runtime (as a function of the polynomial order $N$) for each algorithm (Forsythe and linear system approach) when evaluating the polynomial $P_N^{1,1}(x)$ with the necessary high-precision arithmetic (depicted in the left figure) to achieve the machine precision $\epsilon$. The plot reveals that the Forsythe algorithm requires less time for evaluation, despite needing higher precision numbers.}\label{fig:2_2}
	\end{minipage}
\caption{Comparison of each algorithm in terms of (a) required precision to achieve the output accuracy below machine precision and (b) running time with the required precision. Forsythe algorithm is the appropriate choice if arbitrary precision floating-point numbers are available.}\label{fig:2}
\end{figure}

We now examine the necessary working precision to ensure that the evaluation of $P_N^{k,m}(x)$ attains machine precision for 64-bit floating-point numbers (approximately $\approx10^{-16}$ accuracy). The implementation is based on the Julia programming language \cite{Julia-2017}, utilizing the \texttt{BigFloat} type in conjunction with the \texttt{IntervalArithmetic.jl} package \cite{IntervalArithmetic.jl}. The precision of \texttt{BigFloat} numbers is controlled by the \texttt{setprecision} function, which adjusts the number of significant bits in the floating-point representation\footnote{More specifically, \texttt{setprecision} sets the ``number of bits for significant + 1'', as defined in the Julia language.}.

We selected several random points $x\in[-1,1]$ and adjusted varied the precision of \texttt{BigFloat} to obtain accurate results for the polynomial evaluations. In Figure \ref{fig:2_1}, we present the required precision for each degree $N$ to ensure that the output radius of the interval inclusion is below the machine precision. It is observed that the Forsythe algorithm, which uses recursive formulas, demands significantly higher precision compared to the linear system approach. However, the linear system approach still requires high precision to achieve machine precision, albeit to a lesser extent than the Forsythe algorithm.

In addition, Figure \ref{fig:2_2} illustrates the runtime for evaluating the polynomial with the required precision for each algorithm. Despite the Forsythe algorithm requiring higher precision, it is computationally more efficient and results in faster evaluation times compared to the linear system approach. 

From these observations, we conclude that Forsythe algorithm is the preferred method when arbitrary-precision floating-point numbers are available. This conclusion follows from the computational complexity differences between the algorithms, as well as the capabilities of the Julia environment, where the computation of arbitrary-precision numbers is not prohibitively expensive. Conversely, in environments where interval arithmetic with arbitrary-precision numbers is not supported, the linear system approach is likely to be the more favorable choice.






\section{Zernike Polynomials}
\label{sec:Zernike}

Chebyshev and Legendre polynomials are among the most commonly used orthogonal polynomials on $[-1,1]$, corresponding to specific cases of Jacobi Polynomials (cf. Definition~\ref{def:jacobi_polynomials}) with weights $(-1/2,-1/2)$ and $(0,0)$. 
Other weights $(k,m)$ commonly arise when making a change of coordinates to $[-1,1]$ from another geometry, e.g., spherical or polar coordinates. 
Such is the case for defining a function basis on the disk $ \bD = \{ x^2 + y^2 \leq 1\}$, where typically one takes a polar-radial function basis
\[
    f(\theta,r)=\sum_{m=0}^\infty \left( a_{m} \cos m \theta + b_{m} \sin m \theta \right)  p_{m}(r), \quad p_{m}(r)= \sum_n  c_{m,n} p_{m,n}(r)
\]
for $\theta \in [0,2\pi], r \in [0,1]$, and some basis of radial basis functions $p_{m,n}:[0,1] \to \R$.

There are numerous ways to define the basis functions $p_{m,n}$, each with its own advantages and disadvantages, as discussed in \cite{boyd2011comparing}. A key consideration in this context is the so-called pole condition at $r=0$. Specifically, for a function $f$ to be smooth at $r=0$, the basis function $p_m(r)$ must possess a zero of order $m$ at $r=0$. This condition can either be imposed numerically or inherently satisfied by selecting appropriate basis functions $p_{m,n}$ that naturally fulfill this requirement.

There is no universal ``best'' basis for radial functions. 
While Bessel functions form an eigenbasis for the Dirichlet Laplacian on the disk, their product is generically an infinite series of other Bessel functions, which lacks desirable approximation properties. A commonly used basis for radial functions is given by rescaled Chebyshev functions, which allow efficient evaluation of nonlinearities using the FFT in both variables. Additionally, smooth functions have rapidly converging Chebyshev expansions in this basis. However, this basis does not inherently satisfy the pole condition, and the pseudo-differential operator arising from the Galerkin truncation leads to heptadiagonal matrices, i.e., banded matrices of width seven, as discussed in the works of Boyd \cite{boyd2011comparing} and Shen \cite{shen2000new}.

An effective compromise is provided by the \emph{Zernike polynomials}, also known as \emph{disk polynomials} or \emph{one-sided Jacobi polynomials} and are defined on the unit disk. Zernike polynomials \cite{Zernike,BornWolf1999,Noll1976} are a sequence of orthogonal polynomials  with respect to the natural $L^2$ inner product on the disk, similar to the Legendre polynomials, and possess favorable approximation properties. Named after Frits Zernike, the optical physicist who received the 1953 Nobel Prize in Physics for inventing phase-contrast microscopy, these polynomials are widely used in optics, image analysis, and wavefront fitting. They are particularly useful in describing wavefront aberrations in circular apertures, such as lenses or mirrors in optical systems. Additionally, the natural differential operators on the disk induce relatively simple pseudo-differential operations on the coefficients of the Zernike polynomials. This is analogous to the situation with Chebyshev polynomials of the first kind, where derivatives are most efficiently expressed in terms of Chebyshev polynomials of the second kind, as demonstrated by the identity $ \frac{d}{dx} T_n(x)= n U_{n-1}(x)$. Similarly, an analogous relationship exists for the derivatives of Zernike polynomials.



In the rest of this section, we define formally the Zernike polynomials, demonstrate that the function space of Zernike series polynomials forms a Banach algebra, and apply the MMT approach outlined in Section~\ref{sec:orthogonal_polynomials_and_MMT} to the multiplication of Zernike polynomials.
Evaluating nonlinearities in a Zernike polynomial basis is computationally more expensive than in a Chebyshev polynomial basis, primarily due to the lack of an FFT-like algorithm for Zernike polynomials. However, our MMT approach achieves a computational complexity of \( \cO(N^2) \) with respect to the polynomial degree, as noted in Remark~\ref{rem:CGvsMMT}.

\subsection{Definition of the Zernike polynomials}


As there is not a universally agreed upon notation for Zernike polynomials (e.g. see \cite{koornwinder1978positivity,1985_kanjin_banach_algebra,matsushima1995spectral,boyd2011comparing,vasil2016tensor,arioli2019non}), we now fix our own notation. 

\begin{definition}[\bf Zernike polynomials]
For each  $ k>-1$, and azimuthal wave number $m \in \Z$, we define the $n^{th}$ Zernike polynomial $\mathcal{Q}_n^{k,m}$ with weight $(k,m)$ as
\begin{align}
\mathcal{Q}_n^{k,m}(r,\theta) &\bydef e^{i m \theta } Q_n^{k,m}(r),
&
Q_n^{k,m}(r) &\bydef r^{|m|} P_n^{k,|m|}(x)
\end{align}
for $ x = 2 r^2 -1$ and Jacobi polynomials $P_n^{k,|m|}(x)$, introduced in Definition~\ref{def:jacobi_polynomials}. We also refer to $Q_n^{k,m}(r)$ as the radial Zernike polynomials.
\end{definition}


By writing a complex numbers as $z = r e^{i \theta}$, we may also interpret $ \cQ^{k,m}_n$ as being defined on the complex disk, $ \cQ^{k,m}_n:\bD\to\C$.  In this manner we  may write general functions $ u:\bD\to\C$ as a Zernike polynomials series
\[
       u(r,\theta)= \sum_{m\in \Z,n \geq 0 } {u}_{m,n} \,\cQ_n^{k,m}(r,\theta) 
       = \sum_{m\in \Z  }  \,e^{i m \theta }  \sum_{n\geq 0} {u}_{m,n} Q_n^{k,m}(r) 
\]
which exhibit spectral convergence for analytic functions (e.g. see \cite{Walsh,matsushima1995spectral,boyd2011comparing}).


It is common in the literature to define Zernike polynomials as real-valued functions; however, we opt for a complex-valued definition. This choice mirrors the distinction between Fourier series and trigonometric series. The complex formulation proves particularly advantageous when analyzing the expansion of pseudo-differential operators in terms of Zernike polynomials, as discussed in Section~\ref{sec:PseudoDiff}. For all $m \in \Z$ and $z \in \bD$, we have the relationship 
\begin{equation}\label{eq : complex conjugate of jacobi}
    \mathcal{Q}_n^{k,-m}(z) = \overline{ \mathcal{Q}_n^{k,m}(z)}.
\end{equation}
 When restricting to real-valued functions $u: \bD \to \R$, the coefficients must satisfy $u_{m,n} = \bar{u}_{-m,n}$. Furthermore, by identifying $\C \cong \R^2$, Zernike polynomials offer an elegant representation for vector fields defined on the disk. This representation is particularly useful for translating a vector-valued PDE into its corresponding spectral formulation.

\subsection{Banach Algebra} \label{sec:BanachAlgebra}

Having established the Zernike polynomials, we now formalize a Banach space consisting of rapidly decaying Zernike coefficients of functions expressed as series in terms of Zernike polynomials. We will demonstrate that this sequence space forms a Banach algebra under a suitably chosen weighted norm. This structure will be particularly useful for the nonlinear analysis required in the computer-assisted proofs of existence for solutions to elliptic semilinear PDEs on the disk, as discussed in Section~\ref{sec:PDE}.

\begin{definition}
Fix a Banach space $X$ and a product structure $ * : X \times X \to X$. Then $(X,*)$ is said to be a Banach algebra if for all $a , b \in X$,
\[
\| a * b \|_X \leq \| a\|_X \| b \|_X.
\]
\end{definition}

It is well-established that the space of functions represented as series in Zernike polynomials can be endowed with a Banach algebra structure under an appropriate norm \cite{1985_kanjin_banach_algebra, arioli2019non}. By choosing a suitably weighted norm, this space can be made to correspond to either real analytic functions or weakly differentiable functions. Below, we provide a definition of admissible weighted norms that yield a Banach algebra, thereby generalizing several results previously established in the literature.

\begin{definition} \label{def:AdmissibleWeights}
    We say that a sequence of weights $(w_{m,n})_{(m,n) \in \mathbb{Z}\times \mathbb{N}_0}$  is admissible if: 
\begin{align*}
    w_{m,n} &\geq 1, &
    w_{m,n} &\leq w_{m,n+1}
\end{align*}
and the weights are submultiplicative in the following sense:
\begin{align*}
    w_{m_1+m_2,n'} &\leq w_{m_1,n_1} w_{m_2,n_2}, & n'= n_1+n_2+\frac{|m_1|+|m_2|-|m_1+m_2|}{2}.
\end{align*}
\end{definition}

We note that the radial Zernike polyomials $ Q^{k,m}_n(r)$ are polynomials of degree $ 2n+|m|$, and some possible choices of $w$ which grow geometrically or algebraically are:
\begin{align}
    w_{m,n} &=  \nu^{2n+|m|} ,&  w_{m,n} &=  (1+2 n+ |m|)^s
\end{align}
for $\nu \geq 1$ or $ s \geq 0$.  The proof is left to the reader. 
The geometric weights correspond to analytic functions, whereas the algebraic weights are analogous to a Sobolev norm (e.g. see Definition~\ref{def:WeightedNorms}).

As we will discuss further in Section \ref{sec:PseudoDiff}, the index $k$ in the definition of Zernike polynomials is a grading which essentially corresponds to derivatives, similar to Chebyshev polynomials of the first, second, etc kind. 
If $-1<k\leq0$ then all the polynomials $Q^{0,m}_n(r)$ are bounded between $-1$ and $1$, however in general $\|Q^{k,m}_n\|_{L^\infty}$ grows algebraically in $n$.  
In particular for integers $ k \geq 0$ we have:
\[
\sup_{r \in [0,1]} \left| Q^{k,m}_n(r) \right| = Q^{k,m}_n(1) =
\begin{pmatrix}
		k+n\\
		n
	\end{pmatrix}.
\]
Other works in the literature commonly normalize the basis elements by this factor. We opt to incorporate it into our norm.

\begin{definition} \label{def:WeightedNorms}
	Fix an admissible sequence of  weights $(w_{m,n})_{(m,n) \in \mathbb{Z}\times \mathbb{N}_0}$,  and define the symbol 
 \[
 \langle m,n \rangle_{k,w} = 
 w_{m,n}
		\begin{pmatrix}
			k+n\\
			n
		\end{pmatrix}.
 \]
 Define the weighted sequence space $V^k $ as
	\begin{align}
		V^k &\bydef \{ a=  \{a_{m,n} \}_{m\in\Z,n\in\N} \subseteq \C : \| a \|_{V^k} < \infty \}  ,
		&
		\| a \|_{V^k} &\bydef  \sum_{m\in \Z, n\in \N} 
		|a_{m,n}| \langle m,n \rangle_{k,w}.
	\end{align}
\end{definition}


Analogous to the discrete Fourier transform, we may map a given sequence of Zernike coefficients to the corresponding series of Zernike polynomials. 
Formally, for each $ k \in \N$, define a transform $ \cM : V^{k} \to C( \mathbb{D} , \C)$ given below, where for $ a \in V^{k}$  and  $z=r e^{i \theta} \in \bD$ we define
\[
\cM [a](z) \bydef \sum_{m \in \Z} \sum_{n \in \N}   a_{m,n} \cQ^{k,m}_n (z).
\]
Similarly, as the Zernike polynomials form a Schauder basis, an inverse transform $ \cM^{-1}$ may be defined to recover the Zernike coefficients of a function $u: \bD \to \C$, 
assuming $u$ has sufficient regularity  \cite{Walsh,matsushima1995spectral,vasil2016tensor}. 




Corresponding to the sequence space $V^k$ --- and analogous to the Wiener algebra of functions with absolutely convergent Fourier series --- let us define a space consisting of functions whose Zernike coefficients are in $V^k$. 
\begin{align}
    \tilde{V}^k &= \left\{  u \in L^2( \bD , \C) :  \cM^{-1} (u) \in V^k
\right\},
&
\| u \|_{\tilde{V}^k} = \| \cM^{-1}u \|_{V^k}.
\end{align}
Note that by definition $ \tilde{V}^k$ is isometrically isomorphic with $ V^k$. 
Furthermore, we have the bound $\| u \|_{C(\bD,\C)} \leq \| u\|_{\tilde{V}^k}$. 

In further analogy with the discrete Fourier transform, the product  of functions in $C( \mathbb{D} , \C)$ induces a discrete convolution of their coefficients in $V^k$.

\begin{definition}
	For $ a ,b\in V^{k}$, we define a discrete convolution product $ * : V^{k} \times V^{k} \to V^{k}$ by
	\[
	a * b = \cM^{-1} \left(  \cM[a] \cdot \cM[b]  \right). 
	\]
\end{definition}

We also will consider the subspace $ V^{k,m} \subseteq V^k$ obtained by restricting to a single wave number:
	\begin{align}\label{def : definition of Vkm}
		V^{k,m} &\bydef
		\left\{ a \in V^k : a_{m',n}=0 \mbox{ whenever } m' \neq m
  \right\}.
	\end{align}
As mentioned earlier,  the space of Zernike polynomials forms a Banach algebra for an adequately chosen norm.  
A key element for this proof is the following result on positivity of linearization coefficients:
\begin{lemma} \label{prop:BanachAlgebraRange}
    Fix $ k $ and $ m_1, m_2 \in \Z$. Then  $*$ restricts to a map $ * :V^{k,m_1} \times V^{k,m_2} \to V^{k,m_3}$ for $ m_3 = m_1+m_2$. Furthermore, there exist non-negative linearization coefficients   $c_{(m_1,n_1),(m_2,n_2)}^{(m_3,n_3)}  \geq 0 $  for which 
\begin{align}
    \cQ^{k,m_1}_{n_1}(z)
\cQ^{k,m_2}_{n_2}(z)
= \sum_{ n_3 }
c_{(m_1,n_1),(m_2,n_2)}^{(m_3,n_3)} 
\cQ^{k,m_3}_{n_3} (z) ,
\label{eq:MultiplicationPositivity}
\end{align}
and the linearization coefficients  are non-zero only if 
\[
|n_1 - n_2| -(|m_1|+|m_2|) \leq n_3 \leq n_1+n_2+ \frac{|m_1|+|m_2| -|m_1+m_2|}{2}.
\]
\end{lemma}
\begin{proof}
This essentially follows from  Corollary 5.2  in \cite{koornwinder1978positivity} which yields that the product of two Zernike polynomials may be written as:     
\begin{align}
    \cQ^{k,m_1}_{n_1}(z)
\cQ^{k,m_2}_{n_2}(z)
= \sum_{n_3}
c_{(m_1,n_1),(m_2,n_2)}^{(m_3,n_3)} 
\cQ^{k,m_3}_{n_3} (z), 
\end{align}
where $m_3 =  m_1+m_2 $ and $n_3$ is subject to the constraint:
\begin{align}
    \left|2 n_1 + |m_1| -(2n_2+|m_2|) \right| &\leq 2 n_3 + | m_3| \leq 2 n_1 +|m_1| + 2n_2 + |m_2|.  
\end{align}
This is equivalent to 
\begin{align}
     \left| n_1-n_2 +\frac{|m_1|-|m_2|}{2} \right|-\frac{|m_1+m_2|}{2} &\leq  n_3  \leq  n_1 +n_2 + \frac{|m_1|+|m_2|-|m_1+m_2|}{2},
\end{align}
from which, we obtain the desired result. 
\end{proof}

\begin{theorem}
\label{lem : banach algebra}
The spaces  $ V^k$ and $ \tilde{V}^k$ are  Banach algebras. 
\end{theorem}

\begin{proof}

We will first  show that 
\begin{align}
\label{eq:PairwiseProductBound}
\| \mathcal{Q}^{k,m_1}_{n_1} \mathcal{Q}^{k,m_2}_{n_2}  \|_{\tilde{V}^k} \leq  \| \mathcal{Q}^{k,m_1}_{n_1} \|_{\tilde{V}^k} \| \mathcal{Q}^{k,m_2}_{n_2}  \|_{\tilde{V}^k} 
\end{align}
for any $ m_1,m_2 \in \Z$ and 
 $n_1, n_2 \in \N$. 
 From Lemma \ref{prop:BanachAlgebraRange} it follows  that we may write the product of two Zernike polynomials as 
    \begin{equation} \label{eq:Zernike_property1}
        \mathcal{Q}^{k,m_1}_{n_1}(r,\theta) \mathcal{Q}^{k,m_2}_{n_2}(r,\theta) = 
e^{i(m_1+m_2)\theta}        \sum_{n_3=0}^{n'} c_{(m_1,n_1),(m_2,n_2)}^{(m_3,n_3)} Q^{k,m_1+m_2}_{n_3}(r )
    \end{equation}
    for $n' = n_1+n_2+ \frac{|m_1|+|m_2| -|m_1+m_2|}{2}$ and  real linearization coefficients $c_{(m_1,n_1),(m_2,n_2)}^{(m_3,n_3)} $.  
    Note that as $
Q^{k,m}_n(1) =\begin{pmatrix}
			k+n\\
			n
		\end{pmatrix}
    $ and the coefficients are non-negative, letting $(r,\theta)=(1,0)$ in \eqref{eq:Zernike_property1} above,
    it follows that:  
    \begin{align}\label{eq : property on the cj's coefficients}
    \begin{pmatrix}
			k+n_1\\
			n_1
		\end{pmatrix}
  \begin{pmatrix}
			k+n_2\\
			n_2
		\end{pmatrix}
  &=
        \sum_{n_3 =0}^{n'} 
        \left|
        c_{(m_1,n_1),(m_2,n_2)}^{(m_3,n_3)} 
        \right|
        \begin{pmatrix}
			k+n_3\\
			n_3
		\end{pmatrix}.
    \end{align} 
    Evaluating the norm of the product $\mathcal{Q}^{k,m_1}_{n_1}\mathcal{Q}^{k,m_2}_{n_2}$  and applying  H\"older's inequality  we obtain: 
    \begin{align*}\label{eq : product of two elements}
        \| \mathcal{Q}^{k,m_1}_{n_1}\mathcal{Q}^{k,m_2}_{n_2}\|_{\tilde{V}^k}
        &= \sum_{n_3=0}^{n'} 
        \left| c_{(m_1,n_1),(m_2,n_2)}^{(m_3,n_3)}   \right|
 w_{m_3,n_3}
		\begin{pmatrix}
			k+n_3\\
			n_3
		\end{pmatrix} \\
  &\leq \left(  \sup_{0\leq n_3 \leq n' } w_{m_3,n_3} \right) \left( 
        \sum_{n_3 =0}^{n'} \left| c_{(m_1,n_1),(m_2,n_2)}^{(m_3,n_3)}   \right|
        \begin{pmatrix}
			k+n_3\\
			n_3
		\end{pmatrix} \right) 
  \\
  &=
   w_{m_3,n'} \begin{pmatrix}
			k+n_1\\
			n_1
		\end{pmatrix}
  \begin{pmatrix}
			k+n_2\\
			n_2
		\end{pmatrix}, 
    \end{align*} 
    where the last equality follows from \eqref{eq : property on the cj's coefficients} and the $
    w_{m,n} \leq w_{m,n+1}$ property of the weights. 
    Furthermore as $m_3 = m_1 + m_2$ and the weights are  submultiplicative   $w_{m_1+m_2,n'} \leq w_{m_1,n_1} w_{m_2,n_2}$, we obtain our first result: 
\begin{align}
    \| \mathcal{Q}^{k,m_1}_{n_1}\mathcal{Q}^{k,m_2}_{n_2}\|_{\tilde{V}^k}  
  &\leq  w_{m_1,n_1} w_{m_2,n_2}
   \begin{pmatrix}
			k+n_1\\
			n_1
		\end{pmatrix}
  \begin{pmatrix}
			k+n_2\\
			n_2
		\end{pmatrix} 
  =  \| \mathcal{Q}^{k,m_1}_{n_1} \|_{\tilde{V}^k}  
  \| \mathcal{Q}^{k,m_2}_{n_2}\|_{\tilde{V}^k}  .
\end{align}

To prove the general result, fix functions $u,v \in \tilde{V}^k$ with coefficients $ a,b \in V^k$.  We compute the norm of their product using our bound in \eqref{eq:PairwiseProductBound}. 
\begin{align*}
        \left\| \left( \sum_{m_1,n_1} a_{m_1,n_1} \cQ^{k,m_1}_{n_1} \right) \left( \sum_{m_2,n_2} b_{m_2,n_2} \cQ^{k,m_2}_{n_2} \right) \right\|_{\tilde{V}^k}
   & \leq 
  \sum_{m_1,m_2,n_1,n_2} 
  \left| a_{m_1,n_1}  b_{m_2,n_2} \right| \|  \cQ^{k,m_1}_{n_1} \cQ^{k,m_2}_{n_2} \|_{\tilde{V}^k} \\
  &\leq 
  \sum_{m_1,m_2,n_1,n_2} \left| a_{m_1,n_1}  \right| \|  \cQ^{k,m_1}_{n_1} \|_{\tilde{V}^k}
  \, \left| b_{m_2,n_2}\right| 
  \|  \cQ^{k,m_2}_{n_2} \|_{\tilde{V}^k} \\
  &\leq  \left\| \sum_{m_1,n_1} a_{m_1,n_1} \cQ^{k,m_1}_{n_1} \right\|_{\tilde{V}^k} \left\| \sum_{m_2,n_2} b_{m_2,n_2} \cQ^{k,m_2}_{n_2} \right\|_{\tilde{V}^k}.
\end{align*}
We have thus shown that $ \tilde{V}^k$ is a Banach algebra. As $V^k$ and $ \tilde{V}^k$ are isometrically isomorphic, then $ V^k$ is also a Banach algebra. 
\end{proof}


\subsection{ The MMT operator for Zernike polynomials and evaluating nonlinearities}
\label{sec:MMTforDiskPolynomials1}

In this section we wish to provide a general framework for the product of Zernike series using the MMT approach, which is  essentially the example of evaluating polynomial nonlinearity $\cG$ defined in \eqref{eq:G(f)}.
Since the Zernike polynomials are defined through the orthogonal Jacobi polynomials, we can make use of the theory from the previous sections. There are a couple difficulties. First, each sequence of Zernike polynomials $ \{ \cQ^{k,m}_n\}_{n \in \N}$ uses a different set of interpolation points for each pair $(k,m)$. Second, the function space spanned by the first $N$ polynomials $ \{ \cQ^{k,m}_n\}_{n=0}^{N}$ will have polynomials of degree $ 2N+m$, hence the off the shelf theorems from before cannot be applied. The first issue  is not a big problem if we are dealing with axisymmetric problems, a common simplifying assumption for  PDEs. The second problem is easily dealt with using a  change of variables. 

Consider finite sequences $ a \in V^{k,m_1} $ and $ b \in V^{k,m_2}$ such that $ a_{m_1,n_1}=0$ and $b_{m_2,n_2}=0$ for all $ n_1 >N_1$ and $n_2 > N_2$.  
Suppose we are interested in computing the discrete convolution $ c = a*b \in V^{k,m_3}$ for $m_3 = m_1 + m_2$, corresponding to the product: 
\begin{align}
  \sum_{0 \leq n_3 \leq N'} c_{n_3} Q^{k,m_3}_{n_3}(r)  &=
  \left( \sum_{0\leq n_1 \leq N_1} a_{n_1} Q^{k,m_1}_{n_1} (r) \right) \left( \sum_{0\leq n_2 \leq N_2} b_{n_2} Q^{k,m_2}_{n_2} (r) \right),
\end{align}
where $ N' = N_1 +N_2 +  \frac{|m_1|+|m_2| -|m_1+m_2|}{2}$. 
Recall our definition  of 
$ Q_n^{k,m}(r) = r^{|m|} P_n^{k,|m|}(x)$, where $ x = 2 r^2 -1$ for $ x \in [-1,1]$. 
Grouping the leading powers of $ r$, we obtain:

\begin{align}\label{eq : identity product}
  \sum_{0 \leq n_3 \leq N'} c_{n_3} P_{n_3}^{k,|m_3|}(x)
  &= 
     r^{|m_1|+|m_2|-|m_1+m_2|} 
     \left( \sum_{0\leq n_1 \leq N_1} a_{n_1} P_{n_1}^{k,|m_1|}(x) \right) 
     \left( \sum_{0\leq n_2 \leq N_2} b_{n_2} P_{n_2}^{k,|m_2|}(x) \right). 
\end{align}

Note that $|m_1|+|m_2|-|m_1+m_2|$ is  $0$ if $m_1$ and $m_2$ are of the same sign, otherwise  it is an even integer. More precisely:
\[
\bar{m} \bydef
\frac{|m_1|+|m_2|-|m_1+m_2|}{2} =
\begin{cases}
    0 & \mbox{if } m_1 m_2 \geq 0 \\
     \min \{ |m_1|,|m_2|\} 
    & \mbox{otherwise.}
\end{cases}
\]
Thereby, we can write the leading power of  $r$ as a polynomial in $x$, thus obtaining: 
\begin{align}
\label{eq:Coefficients_Of_A_Product}
  \sum_{0 \leq n_3 \leq N'} c_{n_3} P_{n_3}^{k,|m_3|}(x)
  &= \left(\frac{x+1}{2}\right)^{\bar{m} }
  \left( \sum_{0\leq n_1 \leq N_1} a_{n_1} P_{n_1}^{k,|m_1|}(x) \right) 
  \left( \sum_{0\leq n_2 \leq N_2} b_{n_2} P_{n_2}^{k,|m_2|}(x) \right). 
\end{align}

At this point, we are able to apply the standard interpolation methods to compute the coefficients $(c_{n_3})$ using the MMT and iMMT introduced in Section \ref{sec:orthogonal_polynomials_and_MMT}. 
The paradigm for computing the coefficients $ c_{n_3}$,  as summarized in Figure \ref{fig:CommDiagram}, is to  evaluate the series of $a$ and $b$ converting them into grid-space, perform the multiplication on grid-space, and then convert back to coefficient-space. 

As we are interested in multiplying series of Jacobi polynomials of potentiality different  weights, on the nodes associated with  Jacobi polynomials of potentiality different weights, we may need to use multiple MMT matricies, see Definition \ref{def:MMT}. Let us fix notation below:

\begin{definition}
For a fixed $ N'\in \N$, define  $ M^{(k,m)}_{(k',m')}$ to be the MMT matrix  evaluating the Jacobi polynomials $\{P^{k,m}_n\}_{n=0}^{N'}$  on the nodes  $ \{x_i'\}_{i=0}^{N'}$,  given by  the zeros of the Jacobi polynomial $P^{k',m'}_{N'+1}$. \end{definition}

Note that if $ (k,m)=(k',m')$, then the entries of  the iMMT matrix  $[ M^{(k,m)}_{(k',m')}]^{-1}$ are given by \eqref{eq:iMMT_matrix} given from Theorem \ref{thm:Interpolation}.

The product in \eqref{eq:Coefficients_Of_A_Product} is guaranteed to be a polynomial of degree $N'$, so we may exactly compute its coefficients with quadrature on $N'$ points.  
Then we define the grid-space vectors $ \vec{f}_0 , \vec{f}_1,\vec{f}_2 \in \R^{N'+1}$ as:   
\begin{align}
    (\vec{f}_0)_i &= \left(\frac{x_i+1}{2}\right)^{\bar{m} },&
    (\vec{f}_1)_i&= \sum_{0\leq n_1 \leq N_1} a_{n_1} P_{n_1}^{k,|m_1|}(x_i),&
    (\vec{f}_2)_i&= \sum_{0\leq n_2 \leq N_2} b_{n_2} P_{n_2}^{k,|m_2|}(x_i).
\end{align}
For $\vec{f}_1 $ and $\vec{f}_2 $, this  may also be expressed as a matrix-vector product with the MMT matrix: 
\begin{align}
\vec{f}_1 &= M^{(k,m_1)}_{(k, m_3)} 
\iota^{N'}
\vec{a} ,
&
\vec{f}_2 &= M^{(k,m_2)}_{(k, m_3)} 
\iota^{N'}
\vec{b}, 
\end{align}
for the zero section  $\iota^{N'}: \R^{N+1} \to \R^{N+1} \times \R^{N'-N}$. 
Inherited from the multiplications of functions, we may define a componentwise multiplication product (also called a Hadamard product) $ \odot : \R^{N'} \times \R^{N'}  \to \R^{N'}$ for vectors $ x,y \in \R^{N'}$ in grid space by 
\[
\left( x_1, \dots , x_{N'} \right) \odot 
\left( y_1, \dots ,y_{N'} \right) 
\bydef
\left( x_1 y_1, \dots , x_{N'}  y_{N'}  \right) 
\]
We may compute the output coefficients as 
\begin{align}
    \vec{c} &= [M^{(k,m_3)}_{(k, m_3)} ]^{-1} 
    \left(
    \vec{f}_0  \odot \vec{f}_1  \odot \vec{f}_2 
    \right) 
\end{align}

Overall, we are able to exactly compute the coefficients in \eqref{eq:Coefficients_Of_A_Product}. While there may be rounding error from floating point arithmetic, the effects of this may be bounded using interval arithmetic.

\begin{remark}
If $m_1 m_2 \ge 0$ and $f=f_1=f_2$, then we get that $f_0(x)f_1(x)f_2(x) = (f(x))^2$. This corresponds exactly to the case in Section \ref{sec:Introduction} with $\cG(u)=u^2$. Furthermore, when considering more general (non-polynomial) nonlinearities, aliasing for the Zernike coefficients must be addressed. In particular, handling fractional powers of $r^{|m|}$ for the radial component is not straightforward. We leave this issue for future work.
\end{remark}

We also note that this method of multiplication can be modified to handle functions with different $k$, which is necessary in settings such as the axisymmetric Navier--Stokes equations. In such cases, different values of  $k$  are required. Nevertheless, the method works well if $k_1$, $k_2$, $k_3$ satisfy $ k_3 = \max \{ k_1, k_2\}$ .

Finally, it is also worth mentioning that an alternative method of computing products \eqref{eq:Coefficients_Of_A_Product} may be performed by first converting all of the coefficients to the same set of basis functions, and then computing the product with the MMT. More precisely, suppose now that $a$ is a sequence in the $(P^{k,m_1}_n)_{n \in \mathbb{N}}$ basis, $b$ in the $(P^{k,m_2}_n)_{n \in \mathbb{N}}$ basis and $c$ is the $(P^{k,m_3}_n)_{n \in \mathbb{N}}$ basis where $m_3 = m_1 + m_2$. Under the assumption that $m_1,m_2 \geq 0$, 
we convert both sequences $a$ and $b$ into the basis $(P^{k,m_3}_n)_{n \in \mathbb{N}}$. This can  be achieved through the application of a banded matrix, constructed using to the following relation (see Chapter 18.9 in \cite{nist_handbook})
\[
(2n + k+m+1) P_{n}^{k,m} = (n+k+m+1)P^{k,m+1}_n + (n+k)P^{k,m+1}_{n-1}.
\]
Then, one can evaluate the functions associated to $a$ and $b$ on the $(P^{k,m_3}_n)_{n \in \mathbb{N}}$ grid points using the MMT, $M^{(k,m_3)}_{(k,m_3)}$. After performing a pointwise product, the result on the grid is transformed back to coefficients using the iMMT,  $[M^{(k,m_3)}_{(k,m_3)}]^{-1}$, computed using the formula from Theorem \ref{thm:Interpolation}. 
Indeed, this multiplication method is the one applied in our PDE applications discussed in Section~\ref{sec:PDE}.


\subsection{Pseudo-differential Operators} \label{sec:PseudoDiff}

Before proceeding, we take a moment to clarify the motivation for introducing pseudo-differential operators in this section. These operators are essential for extending the current framework to more complex problems, such as the axisymmetric Navier–Stokes equations or other PDEs on nontrivial geometries. By establishing their properties and interconnections here, we provide a foundation for future research that will build upon the methods and ideas presented. While some of  these operators may seem peripheral to the immediate results of this work, they are critical for appreciating the broader applicability and potential of the techniques discussed. This preparatory step ensures a solid theoretical foundation for applications requiring more advanced analysis.

A central feature of the Zernike polynomials contributing to their effectiveness is that the natural differential operators on the disk yield sparse operators in coefficient space which are easy to invert 
\cite{burns2020dedalus,vasil2016tensor}. 
The  caveat here is that the nice representation of differential operators is only possible if we are willing to change function bases. 
This again is no different than the case with Chebyshev. 
Below we review common linear operators on the spaces of Zernike polynomials from \cite{vasil2016tensor}, having appropriately modified the formulas to account for  the difference in our normalization of basis elements. In particular, we derive such formulas considering $m \geq 0$, since the cases $m<0$ can be derived using \eqref{eq : complex conjugate of jacobi}.

\begin{figure}[h]\centering
\begin{tikzcd}
{(k+1,m-1)} && {(k+1,m)}          && {(k+1,m+1)} \\ \\
{(k,m-1)}   && {(k,m)} \arrow[rr, "R^+"] \arrow[uu, "C" description] \arrow[rruu, "D^+"] \arrow[lluu, "D^-"'] \arrow[ll, "R^-"'] && {(k,m+1)}  
\end{tikzcd}
	\caption{ Various ladder operators on Zernike polynomials, figure emulated from \cite{vasil2016tensor}.}\label{fig:LadderOperators}
\end{figure}
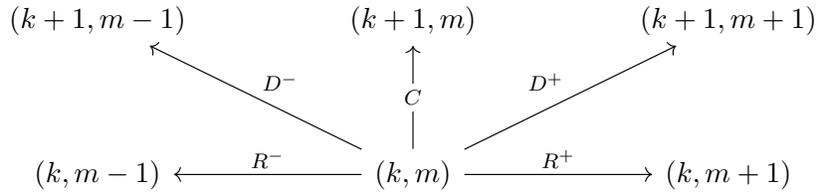

On the two-dimensional disk, there are several common coordinates systems,  such as Cartesian $(x,y)$, polar-radial $(r,\theta)$, or complex $(z,\bar{z})$. 
The Laplacian  on the disk may be expressed in any of these coordinate systems:
\[
\triangle = \partial_x^2 + \partial_y^2 = \partial_r^2 + \tfrac{1}{r} \partial_r + \tfrac{1}{r^2} \partial_\theta^2 = 4 \partial_z \partial_{\bar{z}}
\]
For manipulating the Zernike polynomials, the differential operators $\partial_z $ and $\partial_{\bar{z}}$ have the most natural expression. 
For a motivating example, note that for a complex function $F(r,\theta) = e^{im\theta}f(r)$ with the single azimuthal wave number $m$, then  one may compute its $\partial_z$ and $\partial_{\bar{z}}$ derivatives,  given by 
\begin{align}\label{eq : application of D+ D- on general function}
\nonumber
    \partial_z F(r,\theta) &=\tfrac{1}{2} e^{i(m-1)\theta}\left(\partial_r + \frac{m}{r}\right)f(r)\\
    \partial_{\bar{z}} F(r,\theta) &= \tfrac{1}{2}e^{i(m+1)\theta}\left(\partial_r - \frac{m}{r}\right)f(r).
\end{align}

Notice in particular that the azimuthal wave number is decreased by $\partial_z$, and increased by $  \partial_{\bar{z}} $. 
Thus motivated, we define the differential operators 
\begin{align*}
    D^- &\bydef 2\partial_z,
    &
    D^+ &\bydef 2\partial_{\bar{z}}.
\end{align*}
One may see that these operators commute, and moreover the Laplacian is given by $\triangle = D^+D^- = D^-D^+$. The action of these operators on the Zernike polynomials is in fact easily expressed: each basis element is sent to a scalar multiple of another basis element. 
Specifically, given $(k,m)$, we define 
$D^+_{k,m}:V^{k,m} \to V^{k+1,m+1}$ (respectively $  D^-_{k,m}:V^{k,m} \to V^{k+1,m-1}$) as the restriction of $D^+$ to  $V^{k,m} \to V^{k+1,m+1}$ (respectively of $D^-$ to  $V^{k,m} \to V^{k+1,m-1}$) and obtain the following formulas (see \cite{vasil2016tensor}):
\begin{align*}
    D^+_{k,m} \mathcal{Q}^{k,m}_n &= 2(n+k+m+1)\mathcal{Q}^{k+1,m+1}_{n-1} & \mbox{if }  m&\geq 0, \\
    D^-_{k,m} \mathcal{Q}^{k,m}_n &= 2(n+m)\mathcal{Q}^{k+1,m-1}_n 
  & \mbox{if }  m &>0,
\end{align*}
with the convention that $\mathcal{Q}^{k,m}_{-1} =0$. 
The formulae for $m\leq 0$ may be  easily derived by taking the complex conjugate, see   \eqref{eq : complex conjugate of jacobi}. 
Furthermore, we  deduce the action of the Laplacian $\triangle = D^+D^-$ on Zernike polynomials: 
\begin{align}\label{eq : definition laplacian jacobi basis}
    \triangle \mathcal{Q}^{k,m}_n = 4(n+m)(n+k+m+1)\mathcal{Q}^{k+2,m}_{n-1} \bydef \lambda^{k,m}_n \mathcal{Q}^{k+2,m}_{n-1}. 
\end{align}

Moreover, when combined with Dirichlet boundary condition, we obtain the following inverse for the Laplacian.

\begin{proposition}\label{prop : inverse of dirichlet laplacian}[See \cite{janssen2014zernike}]
Let $m \in \mathbb{N}_0$ and let $\triangle^{-1}_\D$ denote the inverse Laplacian with Dirichlet boundary conditions on the disk. Then, 
{\footnotesize
\begin{align}\label{def : inverse laplacian dirichlet}
\triangle^{-1}_\D \cQ_n^{0,m}  &=  
\begin{dcases}
	\frac{\cQ_{1}^{0,m}}{4(m+1)(m+2)} 
	-
	\frac{\cQ_{0}^{0,m}}{4(m+1)(m+2)} 
&	 \mbox{if } n =0
	\\
	\frac{ \cQ_{n+1}^{0,m}}{4(2n+m+1)(2n+m+2)}
	-
	\frac{ \cQ_{n}^{0,m}}{2(2n+m+2)(2n+m)}
	+
	\frac{\cQ_{n-1}^{0,m}}{4(2n+m)(2n+m+1)} 
	& \mbox{if } n \geq 1
\end{dcases}
\end{align}
}
\end{proposition}

As underlined by \eqref{eq : definition laplacian jacobi basis}, the expression of derivatives has a simple expression if we allow ourselves to change basis functions from $ V^{k}$ to $V^{k+p}$, where $p$ is the degree of differentiation. 
However, when considering a zero-finding problem combining linear differential operators and nonlinearities (cf. Section \ref{sec:PDE}), one often need to change basis in order for the image of the zero-finding problem to be expressed in the same basis.  
For this, one may define a  conversion operator $C : V^{k} \to V^{k+1}$  which converts the $V^k$-coefficients of a function into the $V^{k+1}$-coefficients. In particular,  the restriction $C_{k,m} : V^{k,m} \to V^{k+1,m}$ of $C$ to the subspace $V^{k,m}$ is given as follows (see \cite{vasil2016tensor})
\begin{align} \label{eq:ConversionOp}
	C_{k,m} \cQ_n^{k,m}  &=  
	\frac{n+k+m+1}{2n+k+m+1} \cQ_{n}^{k+1,m} - \frac{n+m}{2n+k+m+1} \cQ_{n-1}^{k+1,m}.
\end{align}
For instance, the operator $\Delta - I : V^{k,m} \to V^{k+2,m}$ restricted to $V^{k,m}$ writes $D^+_{k+1,m-1}D^{-}_{k,m} -  C_{k+1,m}C_{k,m}$ and can be expressed explicitly using the formulas given above.

Following our objective to treat PDEs on the disk, we now introduce the multiplication operators by $z$ and by $\bar{z}$, which appear naturally in polar coordinates. Specifically, we define the operators $R^+$ and $R^-$ as such multiplication operators : 
\begin{align}\label{def : definition of R plus R minus on general function}
R^+ f(z,\bar{z}) &\bydef z f(z,\bar{z}) , &
R^- f(z,\bar{z}) &\bydef \bar{z} f(z,\bar{z}). 
\end{align}
The above operators have an explicit representation when expressed in the Zernike basis. Indeed, given $(k,m)$, we define 
$R^+_{k,m}:V^{k,m}\to V^{k,m+1}$ (respectively $ R^-_{k,m}:V^{k,m}\to V^{k,m-1}$) as the restriction of $R^+$ to  $V^{k,m}\to V^{k,m+1}$ (respectively of $R^-$ to  $V^{k,m}\to V^{k,m-1}$) and obtain the following formulas (see \cite{vasil2016tensor}):
\begin{equation}\label{def : definition of R plus and R minus on jacobi}
\begin{aligned}
	R^+_{k,m} \cQ_n^{k,m}  =  
	\frac{n+k+|m|+1}{2n+k+|m|+1} \cQ_{n}^{k,m+1} + \frac{n+k}{2n+k+|m|+1} \cQ_{n-1}^{k,m+1},  
  \quad \mbox{if } & m\geq 0. \\
	R^-_{k,m} \cQ_n^{k,m}  = 
	\frac{n+1}{2n+k+|m|+1} \cQ_{n+1}^{k,m-1} +
	\frac{n+|m|}{2n+k+|m|+1} \cQ_n^{k,m-1},
  \quad  \mbox{if } & m >0.
\end{aligned}
\end{equation}
Again, the formulae for $m\leq 0$ may be  derived by taking the complex conjugate, see   \eqref{eq : complex conjugate of jacobi}. 
We summarize how the operators $ D^\pm$, $R^\pm$, $C$ all operate between the spaces $ V^{k,m}$  in Figure \ref{fig:LadderOperators}. Having recalled explicit formulas for the different linear operators appearing in the study of PDEs on the disk, we are in a position to investigate specific PDE applications.  

\section{Applications to elliptic semilinear PDEs on the disk} \label{sec:PDE}
 
The framework established in Sections~\ref{sec:orthogonal_polynomials_and_MMT},~\ref{sec:validated_numerics_comparisons}~and~\ref{sec:Zernike} allows the study of elliptic semilinear PDEs on the disk. 
In this section, we leverage these results to produce computer-assisted proofs of the existence of solutions to \eqref{eq : zero finding original} and \eqref{eq : zero finding original 1/z}.

Using the notation established in Section \ref{sec:PseudoDiff}  we may write the PDEs in  \eqref{eq : zero finding original} and \eqref{eq : zero finding original 1/z} for a complex function $ v: \bD \to \C$ as 
\begin{align}\label{eq:LadderEqs}
\begin{cases}
    \triangle v + (R^+)^{-1} v^2 =0,  
    &(\mbox{Case }m=-1)\\
    %
    \triangle v + (R^-)^m v^2 =0, 
    &(\mbox{Case }m\in\N )
\end{cases}
\end{align}
both subject to the Dirichlet boundary conditions $v|_{\partial \bD} =0$.  
Recall from their definition in \eqref{def : definition of R plus R minus on general function} that $ R^+$ and $ R^{-}$ correspond to multiplications by $ z$ and  $ \bar{z}$  respectively. 

In general, the solutions of these PDEs may be written as series in Zernike polynomials $$ v(r,\theta) = \sum  v_{m,n} \cQ_n^{k,m}(r,\theta) = \sum  v_{m,n} e^{i m \theta} Q_n^{k,m}(r) .$$ 
The PDEs do not exactly possess gauge symmetry (i.e. $f( e^{i\theta} v ) = e^{i \theta} f(v)$, for all $\theta$ and $ v\in\C$). 
However, for the $m^{th}$ PDE in the family we are able to make a ``generalized axisymmetric'' assumption  $v(r,\theta) = e^{i m \theta} u(r)$, where $u : [0,1] \to \mathbb{R}$.   Using this ansatz, the equation reduces to an ODE in $r$:
\begin{align}\label{eq : equation in r for u for r to the m}
    \left(
 \partial_r^2 + \tfrac{1}{r} \partial_r - \tfrac{m^2}{r^2} \right)u(r)+r^m u(r)^2 &=0 & r \in [0,1]
\end{align}
paired with the boundary condition $ u(1)=1$. Similarly, the ansatz $v(z) = e^{i \theta} u(r)$ allows to transform \eqref{eq : zero finding original 1/z} into 
\begin{align}\label{eq : equation in r for u for 1/r}
    \left(
 \partial_r^2 + \tfrac{1}{r} \partial_r - \tfrac{1}{r^2} \right)u(r)+ \frac{1}{r}u(r)^2 &=0 & r \in [0,1]
\end{align}
paired with the boundary condition $ u(1)=0$ and the pole condition $u(0) =0.$
This  reduces the problem to looking for a real-valued  solution $u$ to \eqref{eq : equation in r for u for r to the m} and \eqref{eq : equation in r for u for 1/r}  written as a singly indexed series:  
\begin{align}\label{ansatz for u r to the m}
    u(r) =  \sum_{n \in \mathbb{N}_0} u_n Q^{0,|m|}_n(r),
 \end{align}
 where $ u_{n}=v_{|m|,n} \in \mathbb{R}$.  
Note that as $ Q_n^{0,|m|}=r^{|m|}  P_n^{0,|m|}(2r^2-1)$, the pole condition $u^{(k)}(0) = 0$ is satisfied for all $k < |m|$. 

\subsection{Formulation as a \boldmath$F(x)=0$\unboldmath~problem}

In order to apply Theorem \ref{thm::RadiiNonLin}, we will explicitly define a functional equation $ F_m:X \to X'$ whose zeros correspond to solutions of   \eqref{eq:LadderEqs}. 
Resulting from our generalized axisymmetric assumption we fix $X =X'\bydef V^{0,|m|}$, as defined  in \eqref{def : definition of Vkm}, to be the sequence space of Zernike polynomials with fixed wave number  $|m|$.  
The norm on this space is inherited from $V^0$ given in Definition \ref{def:WeightedNorms}, and in this paper we will work with trivial weights (i.e. $\langle m,n\rangle=1$ for all $m,n$). 

 While formally $V^{0,|m|}\subseteq V^0$ consists of doubly indexed sequences $U = (U_{m,n})_{m \in \mathbb{Z}, n \in \mathbb{N}_0}$ where $U_{m',n}=0 $ whenever $m'\neq m$, we may readily  identified the elements with singly indexed  sequences $\tilde{U} \bydef  (U_{m,n})_{n \in \mathbb{N}_0}$. 
In particular, since $k=0$, we might think about $V^{0,|m|}$ as the usual Lebesgue space for sequences $\ell^1(\mathbb{N}_0)$. In order to keep track of the index $m$, we will keep the notation $V^{0,|m|}$ but use the usual notation $\|\cdot\|_1$ for the $1$-norm associated to $V^{0,|m|}$.  
Additionally, for bounded linear operators $ B(V^{0,|m|},V^{0,|m|})$ we will similarly use the notation $ \| \cdot \|_1 $.

Using the notation of Section \ref{sec:PseudoDiff}, we may encode the PDE in \eqref{eq:LadderEqs} as a functional equation in the space of Zernike coefficients.  
One approach to this would be to directly encode the equation; the disk Laplacian $\triangle = D^+ D^-$ has a simple expression as an operator from $ V^{k} $ to $ V^{k+2}$, and an additional equation could be appended to enforce the boundary condition $ 0=u(1) = \sum_{n \in \N} u_n$.  
However it is easier to work with the inverse Dirichlet Laplacian $ \triangle_0^{-1}$, whose action on Zernike polynomials is explicitly given in  
Proposition \ref{prop : inverse of dirichlet laplacian}.

\begin{definition}
    Fix the quadratic nonlinearity $ G:V^0 \to V^0$ as $ G(U)=U*U$ and the inverse Laplacian $ \triangle_0^{-1}$ as in Proposition \ref{prop : inverse of dirichlet laplacian}. We define the family of functions $ F_m : V^{0,|m|} \to V^{0,|m|}$ for each $m \in \mathbb{N}_0 \cup \{-1\}$ as 
    \begin{align}\label{def : definition of Fm}
    F_m(U) \bydef \begin{cases}
         U + \triangle_0^{-1} \left((R^+_{0,1})^{-1} G(U)\right) &\text{ if } m=-1\\
         U + \triangle_0^{-1} \left((R^-)^m G(U)\right) &\text{ if } m \in \mathbb{N}_0.
    \end{cases}
\end{align} 
where $R^+_{0,1}$ is the restriction of $R^{+}$ to $V^{0,1} \to V^{0,2}$.
\end{definition}

 Finally, we  introduce the spaces $V^{k,m}_s$, which are special cases of the ones introduced in Section~\ref{sec:BanachAlgebra} with algebraic weights. Such spaces are natural for the definition of differential operators. Specifically, the index $s$ allows to keep track of the degree of differentiability associated to the series representation of the sequence.

\begin{definition}\label{def : Vk,m,s}
    Define $V^{k,m}_s$ as the special case of the spaces $V^{k,m}$ in \eqref{def : definition of Vkm} for which $\omega_{m,n} = (1+2n+|m|)^s$.  
\end{definition}

In the following lemma, we show that the maps $F_m$ are well defined, and their zeros correspond to solutions of the PDE in \eqref{eq:LadderEqs}.
In particular, the spaces $V^{k,m}_s$  allow us to use a bootstrapping argument, leading to the regularity of solutions.

\begin{lemma}\label{lem : bootstrapping}
Fix $m \in \mathbb{N}_0 \cup\{-1\}$. The function  $ F_m : V^{0,|m|} \to V^{0,|m|}$ is well defined and Fr\'echet differentiable. 
Furthermore if $\widetilde{U} \in V^{0,|m|}$ is a zero of $F_m$, then its corresponding function $\tilde{u}  = \cM[\widetilde{U}] $  is infinitely differentiable and solves the PDE \eqref{eq:LadderEqs}.
\end{lemma}
\begin{proof} 

To show the function is well defined, note first from Proposition \ref{prop:BanachAlgebraRange}  that $G$ restricts to a map from $ V^{0,|m|}$ to $ V^{0,2|m|}$. Furthermore the operators $R^\pm$ restrict to maps from $V^{0,m'}$ to $V^{0,m'\pm1}$ for all $m' \geq 0$. It follows that $ F_m $ is well defined as a map from $ V^{0,|m|}$ to $ V^{0,|m|}$, at least for $m\geq0$.

In the case of $m=-1$, there is an added complication of the operator $ (R^+_{0,1})^{-1}$, which is unbounded and only densely defined. 
Nevertheless,  we show in  Lemma \ref{lem : inverse of R01} that $R_{0,1}^{+} : V^{0,1} \to V^{0,2}$ has a bounded inverse as a map between differently weighted spaces, in particular   that $(R^+_{0,1})^{-1}:V^{0,2}_s\to V^{0,1}_{s-1}$ is a bounded linear operator for all $ s \geq 0$.

From Proposition \ref{prop : inverse of dirichlet laplacian} we have that $\triangle_0^{-1} : V^{k,|m|}_{s} \to V^{k,|m|}_{s+2}$ is a bounded linear operator, and from Lemma \ref{lem : banach algebra} we have that  $G:V_s^{k,|m|} \to V_s^{k,2|m|}$ is smooth   with Frech\'et derivative $ DG(x)h = 2 x*h$ for all $x,h \in V_s^{k,|m|}$. 
It then follows that  the compositions $\triangle_0^{-1} (R^+_{0,1})^{-1} G(U):V^{k,|m|}_{s}  \to V^{k,|m|}_{s+1} $ 
and $\triangle_0^{-1} (R^-)^{m} G(U):V^{k,|m|}_{s}  \to V^{k,|m|}_{s+2} $ are well defined and Frech\'et differentiable, as are all functions $ F_m : V^{0,|m|} \to V^{0,|m|}$. 

The solutions being infinitely smooth follows from a standard bootstrapping argument: 
 make the inductive  assumption that $\widetilde{U} \in V_s^{0,|m|}$. 
If $F(\widetilde{U})=0$, then  $ \widetilde{U} = - \triangle_0^{-1}  (R^+_{0,1})^{-1} G(\widetilde{U})  $ in the case $m=-1$, and $\widetilde{U}=- \triangle_0^{-1}  (R^-)^m G(\widetilde{U}) $ in the case $m\geq 0$. 
By the preceeding argument we demonstrated that  the right hand side of these equations is smoothing, hence $\widetilde{U} \in V_{s+1}^{0,|m|}$ in the case $m=-1$, and  $\widetilde{U} \in V_{s+2}^{0,|m|}$ in the case $m\geq 0$. 
Since  $\widetilde{U} \in V_0^{0,|m|}$ by the hypothesis of this lemma, it follows by induction that $ \widetilde{U} \in V_{s}^{0,|m|}$ for all $ s\geq 0$.  

       Furthermore, note that any derivatives of  $ \tilde{u} = \cM(  \widetilde{U})$ may be obtained as a combination of  $ D^- = 2 \partial_z$ and $ D^+ = 2 \partial_{\bar{z}}$. 
        Since $ \widetilde{U} \in V_{s}^{0,|m|}$ for all $ s\geq 0$, and the  ladder operators 
$           D^{\pm}_{k,m}:V_{s}^{k} \to V_{s-2}^{k+1}$ are bounded, it follows that the derivatives of $ \tilde{u}$ will have bounded Zernike coefficients in $V_{s}^{k}$ for all $ s \geq 0$. Hence $ \tilde{u} \in C^{\infty}(\bD ,\C)$.  As $\tilde{u}$ is smooth it satisfies \eqref{eq:LadderEqs} in the strong sense.
Lastly, as $ \widetilde{U} = - \triangle_0^{-1}  (R^+_{0,1})^{-1} G(\widetilde{U})  $ in the case $m=-1$  (respectively $\widetilde{U}=- \triangle_0^{-1}  (R^-)^m G(\widetilde{U}) $ in the case $m\geq 0$), we see that $\widetilde{U}$ is in the image of the inverse Dirichlet Laplacian, and thereby satisfies Dirichlet-0 boundary conditions. 
\end{proof}

\subsection{Computer-assisted analysis}\label{sec : computer assisted analysis}

Lemma~\ref{lem : bootstrapping} shows that zeros of $F_m$ in $V^{0,|m|}$ provide smooth solutions to \eqref{eq : zero finding original} and \eqref{eq : zero finding original 1/z}. We are now in a position to present our computer-assisted approach to prove the existence of solutions to \eqref{def : definition of Fm}.
Our goal is to use Theorem \ref{thm::RadiiNonLin} in order to prove the existence of a zero of $F_m$ given in \eqref{def : definition of Fm}. In particular, we assume that we have access to some  $U_0 \in V^{0,|m|}$, which is an approximate zero of $F_m$ and $A : V^{0,|m|} \to V^{0,|m|}$, which is a bounded linear operator approximating the inverse of $DF_m(U_0).$ In practice, both $U_0$ and $A$ will be constructed numerically.

Indeed, let us fix $N \in \mathbb{N}$ representing the size of the numerical truncation and define the following projection operators
 \begin{align}\label{def : projection truncation}
 \nonumber
    \pi^N(V)  =  \begin{cases}
          v_n,  & n \in I^N \\
              0, &n \notin I^N
    \end{cases} ~~ \text{ and } ~~
     \pi^{\infty}(V)  =  \begin{cases}
          0,  & n \in I^N \\
              v_n, &n \notin I^N
    \end{cases}
 \end{align}
    where $I^N \bydef \{n \in {\mathbb{N}_{0}}, n \leq N\}$ for all $V = (v_n)_{n \in  {\mathbb{N}_{0}}} \in V^{0,|m|}$.
    In particular, numerically we build $U_0$ such that $U_0 = \pi^{N} U_0$, meaning that $U_0$ only has a finite number of non-zero coefficients ($U_0$ can be seen as a vector) which is stored on the computer. Its construction is obtained using a Newton method (see \cite{julia_zernike} for additional details).




\subsubsection{Construction of the approximate inverse $A$}

Now that the approximate solution $U_0$ is fixed, we present the construction of an approximate inverse $A$ for $DF_m(U_0)$.    
First, noticing that $DF_m(U_0)$ is a compact perturbation of the identity, we look for $A$ as a finite-rank perturbation of the identity as well. In other words, we numerically construct $A^N$ satisfying $A^N = \pi^N A^N \pi^N$ ($A^N$ can be seen as a $N+1$ by $N+1$ matrix) and such that $A^N$ approximates the inverse of the Galerkin projection of $DF_m(U_0)$ of size $N$, namely $\pi^NDF_m(U_0)\pi^N$. Then we define $A : V^{0,|m|} \to V^{0,|m|}$ as 
\[
    A \bydef A^N + \pi^{\infty}.
\]
Intuitively, the tail of the operator $A$ acts as the identity. We will show in Lemma \ref{lem : Z1 bound} that this construction is justified when $N$ is big enough.

Now that $U_0$ and $A$ are constructed, following Theorem \ref{thm::RadiiNonLin}, it remains to compute the bounds $Y_0, Z_1$ and $Z_2$ introduced in the theorem.

 \subsubsection{Computation of the bounds}

\begin{lemma}\label{lem : Y0 bound}
    Let $Y_0$ satisfying 
    \begin{align}
    Y_0 \geq 
 \begin{cases}
      \|A^N(U_0 + \pi^N\triangle_{\D}^{-1}(R^{-})^m(U_0*U_0)\|_1 + \|(\pi^{2N+m+1}-\pi^N) \triangle_{\D}^{-1}(R^{-})^m(U_0*U_0)\|_1 &\text{ if } m \in \mathbb{N}_0\\
       \|A^N(U_0 + \pi^N\triangle_{\D}^{-1}(R_{0,1}^{+})^{-1}(U_0*U_0)\|_1 + \|(\pi^{2N+1}-\pi^N) \triangle_{\D}^{-1}(R_{0,1}^{+})^{-1}(U_0*U_0)\|_1 &\text{ if } m = -1
\end{cases}
\end{align}
    then $\|AF_m(U_0)\|_1 \leq Y_0$. 
\end{lemma}

\begin{proof}
Suppose that $m \in \mathbb{N}_0$. Then, notice that since $U_0 = \pi^N U_0$, then $U_0$ represents a polynomial of order $2N+m$ and consequently, $U_0*U_0$ represents a polynomial of order $4N+2m$. This implies that $U_0*U_0 = \pi^{2N}U_0*U_0$, when $U_0*U_0$ is written in $V^{0,2m}$.  Then, as $(R^{-})^m$ is lower triangular and banded, with a band of size $m$, we obtain $(R^{-})^m(U_0*U_0) = \pi^{2N+m}(R^{-})^m(U_0*U_0)$.  Finally, using \eqref{def : inverse laplacian dirichlet}, we have $\triangle_{\D}^{-1}(R^{-})^m(U_0*U_0) = \pi^{2N+m+1}\triangle_{\D}^{-1}(R^{-})^m(U_0*U_0)$. Therefore, we have
    \begin{align*}
        \|AF_m(U_0)\|_1 &=  \|A^N(U_0 +  \triangle_{\D}^{-1}(R^{-})^m(U_0*U_0)\|_1 + \| \pi^{\infty} \triangle_{\D}^{-1}(R^{-})^m(U_0*U_0)\|_1\\
        &=    \|A^N(U_0 +  \pi^N\triangle_{\D}^{-1}(R^{-})^m(U_0*U_0)\|_1 + \| (\pi^{2N+m+1}-\pi^N) \triangle_{\D}^{-1}(R^{-})^m(U_0*U_0)\|_1.
    \end{align*}

Suppose now that $m=-1.$ Then, as $(R_{0,1}^{+})^{-1}$ is upper triangular (cf. \eqref{def : 1/r}), we obtain that $(R_{0,1}^{+})^{-1}(U_0*U_0) = \pi^{2N}(R_{0,1}^{+})^{-1}(U_0*U_0)$.  Moreover, using \eqref{def : inverse laplacian dirichlet}, we have $\triangle_{\D}^{-1}(R_{0,1}^{+})^{-1}(U_0*U_0) = \pi^{2N+1}\triangle_{\D}^{-1}(R_{0,1}^{+})^{-1}(U_0*U_0)$. We conclude the proof in a similar fashion as the case $m \in \mathbb{N}_0.$  
\end{proof}

\begin{lemma}\label{lem : Z1 bound}
    Let $Z_{0,m}$ be a bound satisfying
    \begin{align*}
         Z_{0,m} &\geq 
    \begin{cases}
   \|\pi^N - A^N(I + \triangle_{\D}^{-1}(R^{-})^mDG(U_0)\pi^{N})\|_1 &\text{ if } m \in \mathbb{N}_0\\
     \|\pi^N  - A^N(I + \triangle_{\D}^{-1}(R_{0,1}^{+})^{-1}DG(U_0))\pi^N\|_1 &\text{ if } m = -1.
     \end{cases}
    \end{align*}
    If $Z_1$ satisfies
    \begin{align*}
    Z_1 &\geq 
    \begin{cases}
   \max\left\{Z_{0,m}, ~ \|A^N\triangle_{\D}^{-1}(R^{-})^mDG(U_0)(\pi^{2N+m+1}-\pi^{N})\|_1 \right\}
    + \frac{2\|U_0\|_1}{(2(N+1)+m)^2} &\text{ if } m \in \mathbb{N}_0\\
     \max\left\{Z_{0,-1},~ \|A^N\triangle_{\D}^{-1}DG\left((R_{0,0}^{+})^{-1}U_0\right)(\pi^{2N+1}-\pi^N)\|_1 \right\}
    + \frac{\|(R_{0,0}^+)^{-1}U_0\|_1}{2N^2} &\text{ if } m = -1
     \end{cases}
    \end{align*}
    then $\|I - ADF_m(U_0)\|_1 \leq Z_1.$ 
\end{lemma}

\begin{proof}
Suppose that $m \in \mathbb{N}_0$. Then,
\begin{align*}
    \|I - ADF_m(U_0)\|_1 &= \|\pi^N + \pi^{\infty} - (A^N+ \pi^{\infty})(I + \triangle_{\D}^{-1}(R^{-})^mDG(U_0))\|_1\\
    &\leq  \|\pi^N  - A^N(I + \triangle_{\D}^{-1}(R^{-})^mDG(U_0))\|_1 + \|\pi^{\infty}\triangle_{\D}^{-1}(R^{-})^mDG(U_0)\|_1.
\end{align*}
 Now, notice that because $U_0 = \pi^N U_0$, then 
\[
\pi^N\triangle_{\D}^{-1}(R^{-})^mDG(U_0) = \pi^N\triangle_{\D}^{-1}(R^{-})^mDG(U_0)\pi^{2N+m+1},
\]
using the same reasoning as for the proof of Lemma \ref{lem : Y0 bound}.
Moreover, we have
\begin{align*}
    &\|\pi^N - A^N\triangle_{\D}^{-1}(R^{-})^mDG(U_0)\pi^{2N+m+1}\|_1 \\ \leq & \max\left\{\|\pi^N - A^N\triangle_{\D}^{-1}(R^{-})^mDG(U_0)\pi^{N}\|_1, ~ \|A^N\triangle_{\D}^{-1}(R^{-})^mDG(U_0)(\pi^{2N+m+1}-\pi^{N})\|_1\right\}.
\end{align*}
Finally, notice from \eqref{def : inverse laplacian dirichlet} that we have the operator bound:
{\small
\begin{align}
     \|\pi^{\infty}\triangle_{\D}^{-1} \|_1 &= \sup_{n \geq N+1}  \left\|
     \frac{ \pi_{\infty}\cQ_{n+1}^{0,m}}{4(2n+m+1)(2n+m+2)}
	-
	\frac{ \pi_{\infty}\cQ_{n}^{0,m}}{2(2n+m+2)(2n+m)}
	+
	\frac{\pi_{\infty} \cQ_{n-1}^{0,m}}{4(2n+m)(2n+m+1)} \right\|_1 \\
    &\leq \frac{1}{(2(N+1)+m)^2},
\end{align}
}
and thereby we obtain: 
\begin{align*}
    \|\pi^{\infty}\triangle_{\D}^{-1}(R^{-})^mDG(U_0)\|_1 \leq \|\pi^{\infty}\triangle_{\D}^{-1}\|_1 \|(R^{-})^mDG(U_0)\|_1
     \leq \frac{1}{{(2(N+1)+m)^2}}\|(R^{-})^m\|_1\|2U_0\|_1.
\end{align*}
We conclude the proof of the case $m \in \mathbb{N}_0$ using \eqref{eq : op norm rm}. 

Suppose now that $m=-1$. We have 
\begin{align*}
    \|I - ADF_m(U_0)\|_1 &= \|\pi^N + \pi^{\infty} - (A^N+ \pi^{\infty})(I + \triangle_{\D}^{-1}(R_{0,1}^{+})^{-1}DG(U_0))\|_1\\
    &\leq  \|\pi^N  - A^N(I + \triangle_{\D}^{-1}(R_{0,1}^{+})^{-1}DG(U_0)) + \pi^{\infty}\triangle_{\D}^{-1}(R_{0,1}^{+})^{-1}DG(U_0)\|_1\\
    & \leq \|\pi^N  - A^N(I + \triangle_{\D}^{-1}(R_{0,1}^{+})^{-1}DG(U_0)) \|_1+ \|\pi^{\infty}\triangle_{\D}^{-1}(R_{0,1}^{+})^{-1}DG(U_0)\|_1.
\end{align*}
Now,   given $W \in V^{0,1}$, notice that 
\begin{equation}\label{eq : switch of 1/r}
   (R_{0,1}^{+})^{-1}(U_0*W) = \left((R_{0,0}^{+})^{-1}U_0\right)*W 
\end{equation}
since $\frac{1}{r}(u_0 w) = \frac{u_0}{r} w$ if $w$ is the function representation of $W$. Consequently, we obtain the following 
\begin{align*}
\|\pi_\infty\triangle_{\D}^{-1}(R_{0,1}^{+})^{-1}DG(U_0)\|_1
\leq \|\pi_\infty\triangle_{\D}^{-1} \|_1 \| (R_{0,1}^{+})^{-1}DG(U_0)\|_1  &= 
\frac{1}{4 N^2}  \sup_{\|h\| =1} \| (R_{0,1}^{+})^{-1} ( 2 U_0 * h)  \|_1
\\  &= 
\frac{1}{2 N^2} 
\|  (R_{0,0}^{+})^{-1}   U_0 \|_1.
\end{align*}
Then,
\begin{align*}
  &~~~~  \|\pi^N  - A^N(I + \triangle_{\D}^{-1}(R_{0,1}^{+})^{-1}DG(U_0)) \|_1 \\
    &= \max\left\{\|\pi^N  - A^N(I + \triangle_{\D}^{-1}(R_{0,1}^{+})^{-1}DG(U_0))\pi^N\|_1,~ \|A^N \triangle_{\D}^{-1}(R_{0,1}^{+})^{-1}DG(U_0)\pi^{\infty}\|_1 \right\}.
\end{align*}
Finally, using \eqref{eq : switch of 1/r}, we have that 
\[
(R_{0,1}^{+})^{-1}DG(U_0) = DG\left((R_{0,0}^{+})^{-1}U_0\right)
\]
since $G$ is quadratic. Now, using \eqref{def : 1/r 0}, we have that   $(R_{0,0}^{+})^{-1}$ is upper triangular and therefore $(R_{0,0}^{+})^{-1}U_0 = \pi^N (R_{0,0}^{+})^{-1}U_0$. This yields
\begin{align*}
     \|A^N\triangle_{\D}^{-1}(R_{0,1}^{+})^{-1}DG(U_0)\pi^{\infty}\|_1 &= \|A^N\triangle_{\D}^{-1}DG\left((R_{0,0}^{+})^{-1}U_0\right)\pi^{\infty}\|_1 \\
     &= \|A^N\triangle_{\D}^{-1}DG\left((R_{0,0}^{+})^{-1}U_0\right)(\pi^{2N+1}-\pi^N)\|_1. \qedhere
\end{align*}

%

%
\end{proof}

\begin{lemma}\label{lem : Z2 bound}
    Let $Z_2$ be a bound satisfying 
   \begin{align*}
     Z_2 \geq 
       \begin{cases}
           2\|A^N\triangle_{\D}^{-1}(R^{-})^m\pi^{N+1}\|_1 + \frac{2}{{(2(N+1)+m)^2}} &\text{ if } m \in \mathbb{N}_0\\
            2\|A^N\triangle_{\D}^{-1}(R_{0,1}^{+})^{-1}\pi^{N}\|_1 + {\frac{16}{15}}\frac{\|A^N\|_1}{N+3}+  \frac{1}{2N}  &\text{ if } m = -1
       \end{cases}
   \end{align*}
    then $\|A(DF_m(c)-DF_m(U_0)\|_1 \leq Z_2r$ for all $c \in \overline{B_r(U_0)}.$
\end{lemma}

\begin{proof}
Along the proof, given Banach spaces $X,Y$, we denote $\|\cdot\|_{B(X,Y)}$ as the operator norm for bounded linear operators on $X \to Y.$
Suppose that $m \in \mathbb{N}_0$. Then, given $c \in \overline{B_r(U_0)}$, we have
\begin{align*}
    \|A(DF_m(c)-DF_m(U_0)\|_1& = \|A\triangle_{\D}^{-1}(R^{-})^mDG(c-U_0)\|_1\\
    & \leq \|A^N\triangle_{\D}^{-1}(R^{-})^mDG(c-U_0)\|_1 + \|\pi^{\infty}\triangle_{\D}^{-1}(R^{-})^mDG(c-U_0)\|_1,
\end{align*}
where we used that $DG(c) - DG(U_0) = DG(c-U_0)$ since $G$ is quadratic.
 Now, notice that
\[
\|\pi^{\infty}\triangle_{\D}^{-1}(R^{-})^mDG(c-U_0)\|_1 \leq \frac{2}{{(2(N+1)+m)^2}}\|(R^{-})^m\|_1 \|c-U_0\|_1 \leq \frac{1}{2N^2}\|(R^{-})^m\|_1r
\]
and 
\begin{align*}
\|A^N\triangle_{\D}^{-1}(R^{-})^mDG(c-U_0)\|_1 &= \|A^N\triangle_{\D}^{-1}\pi^{N+1}(R^{-})^mDG(c-U_0)\|_1 \\
& \leq 2\|A^N\triangle_{\D}^{-1}(R^{-})^m\pi^{N+1}\|_1r.
\end{align*}
We conclude the proof of the case $m \in \mathbb{N}_0$ using \eqref{eq : op norm rm}. 

For the case $m=-1$, we have
\begin{align*}
    \|A(DF_m(c)-DF_m(U_0) )\|_1& = \|A\triangle_{\D}^{-1}(R_{0,1}^{+})^{-1}DG(c-U_0)\|_1\\
    & \leq \|A^N\triangle_{\D}^{-1}(R_{0,1}^{+})^{-1}DG(c-U_0)\|_1 + \|\pi^{\infty}\triangle_{\D}^{-1}(R_{0,1}^{+})^{-1}DG(c-U_0)\|_1.
\end{align*}
 Now,  combining Definition \ref{def : Vk,m,s} with \eqref{def : inverse laplacian dirichlet}, we obtain that
 
 {\footnotesize
\begin{align}
     \|\pi^{\infty}\triangle_{\D}^{-1} \|_{B(V^{0,1}_{-1}, V^{0,1}_{0})} \leq 
     \sup_{n \geq N+1} 
     \frac{1}{(2n+2)^{-1}}
     \left(
     \frac{ 1}{4(2n+2)(2n+3)}
	+
	\frac{ 1}{2(2n+3)(2n+1)}
	+
	\frac{1}{4(2n+1)(2n+2)} \right)  \leq \frac{1}{2N}
\end{align}
}
 Then, using \eqref{eq : norm identity inv R}, we get 
\[
\|\pi^{\infty}\triangle_{\D}^{-1}(R_{0,1}^{+})^{-1}DG(c-U_0)\|_1 \leq \|\pi^{\infty}\triangle_{\D}^{-1}\|_{B(V^{0,1}_{-1}, V^{0,1}_{0})}\|(R_{0,1}^{+})^{-1}\|_{B(V^{0,2}_0, V^{0,1}_{-1})}\|DG(c-U_0)\|_1   \leq \frac{1}{2N}r.
\]
Now,  notice that we have 
\begin{align*}
    \|A^N\triangle_{\D}^{-1}(R_{0,1}^{+})^{-1}\pi^{\infty} DG(c-U_0)\|_1 &\leq \|A^N\|_1 \|\triangle_{\D}^{-1}\|_{B(V^{0,1}_{-2},V^{0,1}_0)} \|(R_{0,1}^{+})^{-1}\pi^{\infty}\|_{B(V^{0,2}_0,V^{0,1}_{-2})} \|DG(c-U_0)\|_1\\
    &\leq  2\|A^N\|_1 \|\triangle_{\D}^{-1}\|_{B(V^{0,1}_{-2},V^{0,1}_0)}  \|(R_{0,1}^{+})^{-1}\pi^{\infty}\|_{B(V^{0,2}_0,V^{0,1}_{-2})} r.
\end{align*}
Then using \eqref{eq : appendix lemma}, we get
\begin{align*}
    \|(R_{0,1}^{+})^{-1}\pi^{\infty}\|_{B(V^{0,2}_0,V^{0,1}_{-2})} = \max_{j > N} \left\{ \sum_{i=0}^j \frac{1}{2(j+1)(j+2)}\right\} = \frac{1}{2(N+3)}.
\end{align*}
Also we have the bound 
{
\footnotesize
\begin{align}
 \|\triangle_{\D}^{-1}\|_{B(V^{0,1}_{-2},V^{0,1}_0)}   \leq \max \left\{ \frac{1}{3}, \sup_{n\geq 1} \frac{1}{(2n+2)^{-2}
    } \left(
     \frac{ 1}{4(2n+2)(2n+3)}
	+
	\frac{ 1}{2(2n+3)(2n+1)}
	+
	\frac{1}{4(2n+1)(2n+2)} \right)  
 \right\} \leq \frac{16}{15}.
\end{align} }
Consequently, using a similar reasoning as what was achieved above, we get
\begin{align*}
\|A^N\triangle_{\D}^{-1}(R_{0,1}^{+})^{-1}DG(c-U_0)\|_1  &\leq \|A^N \triangle_{\D}^{-1}(R_{0,1}^{+})^{-1}\pi^N DG(c-U_0)\|_1+ \|A^N\triangle_{\D}^{-1}(R_{0,1}^{+})^{-1}\pi^{\infty} DG(c-U_0)\|_1\\
&\leq 2\|A^N\triangle_0^{-1}(R_{0,1}^{+})^{-1}\pi^N\|_1r + {\frac{16}{15}}\frac{\|A^N\|_1}{N+3}r. \qquad \qedhere
\end{align*}

\end{proof}

\subsection{Existence proofs for different values of \boldmath$m$\unboldmath}\label{sec : existence proofs for different values of m}

In this section we present some computer-assisted proofs of existence of non-trivial solutions to \eqref{eq : zero finding original}  with $m = -1,0,1,2,20$. For each value of $m$, we denote by $U_{0,m}$, the associated approximate solution used in the proof and provide its function representation (real part) in the figures below. The approximate solution is obtained numerically using the procedure described in \cite{julia_zernike}.  Moreover, we provide the numerical quantities related to the proof in Table \ref{table : values of rm}.

In the proofs, we use extended precision to compute and store the MMT and iMMT matricies. Specifically, we use 128 bits to represent the matrices thanks to the ``BigFloat" type in Julia. This extension of the precision allows us to compute products with high accuracy. The rest of the computations are run in standard double floating point arithmetic.

\begin{theorem} \label{thm:disk_solutions}
    Let $m \in \{-1, 0, 1 ,2 ,20\}$ and let $r_m >0$ be its associated value as given in Table \ref{table : values of rm}. Then, there exists $\widetilde{U}_m \in V^{0,|m|}$, such that $\widetilde{U}_m$ solves \eqref{def : definition of Fm} and $\widetilde{U}_m \in \overline{B_{r_m}(U_{0,m})}$.
\end{theorem}
\begin{table}[H]
\centering
\caption{Verified quantities of Theorem \ref{thm::RadiiNonLin}}
\label{table : values of rm}
\begin{tabular}{|c|c|c|c|c|c|c|}
 \hline
 Value of $m$& $Y_0$ & $Z_1$ & $Z_2$ & $N$ & $\|U_{0,m}\|_1$ & Value of $r_m$\\
 \hline
 \hline
 $-1$  & $1.57 \times 10^{-14}$   & $1.49 \times 10^{-1}$   &  $0.46$   & 36 & $17$ &  $1.85\times 10^{-14}$\\
 $0$  & $7.62\times 10^{-16}$   & $4.28 \times 10^{-3}$   &  $0.71$   & 36 & $8.6$ &   $7.65\times 10^{-16}$  \\
  $1$ & $2.89\times 10^{-14}$ & $2.36 \times 10^{-2}$   &  $0.14$  & 36 & $59$ & $2.96\times 10^{-14}$\\
  $2$ & $5.34 \times 10^{-14}$ &  $5.61 \times 10^{-2}$   & $0.058$   & 36 & $151$ & $5.66\times 10^{-14}$\\
  $20$ & $1.52\times 10^{-7}$ & $8.86 \times 10^{-1}$   &  $0.0013$  & 75 & $9900$ & $1.33\times 10^{-6}$ 
  \\
 \hline
\end{tabular}
\end{table}

Notice that for positive values of  $m$, the maximal amplitude increases as $m$ increases. In particular, for $m=20$, we need to increase the number of coefficients in order to approximate accurately the solution and obtain a proof of existence. In particular, this has a detrimental effect on the quality of the radius $r_{20}$ compared to the other solutions. Note, however, that this could be resolved by performing the whole computer-assisted proof in extended precision (whereas, for now, only products are implemented in extended precision).

\section*{Conclusion and Future Work}

In this manuscript, we introduced a computer-assisted framework for the rigorous evaluate of polynomial nonlinearities at series expansions of orthogonal polynomials. Specifically, this methodology is applied to Zernike polynomials, enabling the analysis of nonlinear PDEs on the disk. Several extensions of this work are possible. For instance, one could consider tensor products of Zernike polynomials with Fourier series, which would naturally facilitate the study of PDEs on cylindrical domains with periodic boundary conditions in the vertical direction. Such spatial domains have been explored, for example, in the context of the axisymmetric Navier–Stokes equations \cite{hou2021nearly, hou2023potential}. On the technical side, however, this tensor product introduces more intricate operator structures on sequences, which would require detailed analysis to enable computer-assisted proofs. This extension is currently under investigation.

Another promising direction for future work is the treatment of nonpolynomial nonlinearities. For instance, extending our methodology to PDEs with exponentials in the nonlinearity --- such as the Karma system, which is  used to model spiral wave breakup
in cardiac tissue \cite{karma1993spiral,karma1994electrical,dodson2019determining} --- would require adaptations. Specifically, given a finite polynomial  $u(x) = \sum_{n=0}^N a_n p_n(x)$, where $\{p_n\}$ is a family of orthogonal polynomials, the nonlinear transformation $\mathcal{G}(u)$ yields an infinite sequence of coefficients, as opposed to the finite representation possible for polynomial nonlinearities. To address this, one would need to develop a Poisson-like formula, akin to those used in conjunction with the Fourier Discrete Transform, and employ precise techniques from complex analysis to rigorously estimate the coefficients of $\mathcal{G}(u)$.







\begin{figure}[htbp]\centering
	\begin{minipage}[b]{0.48\linewidth}
		\centering
		\includegraphics[width=\textwidth]{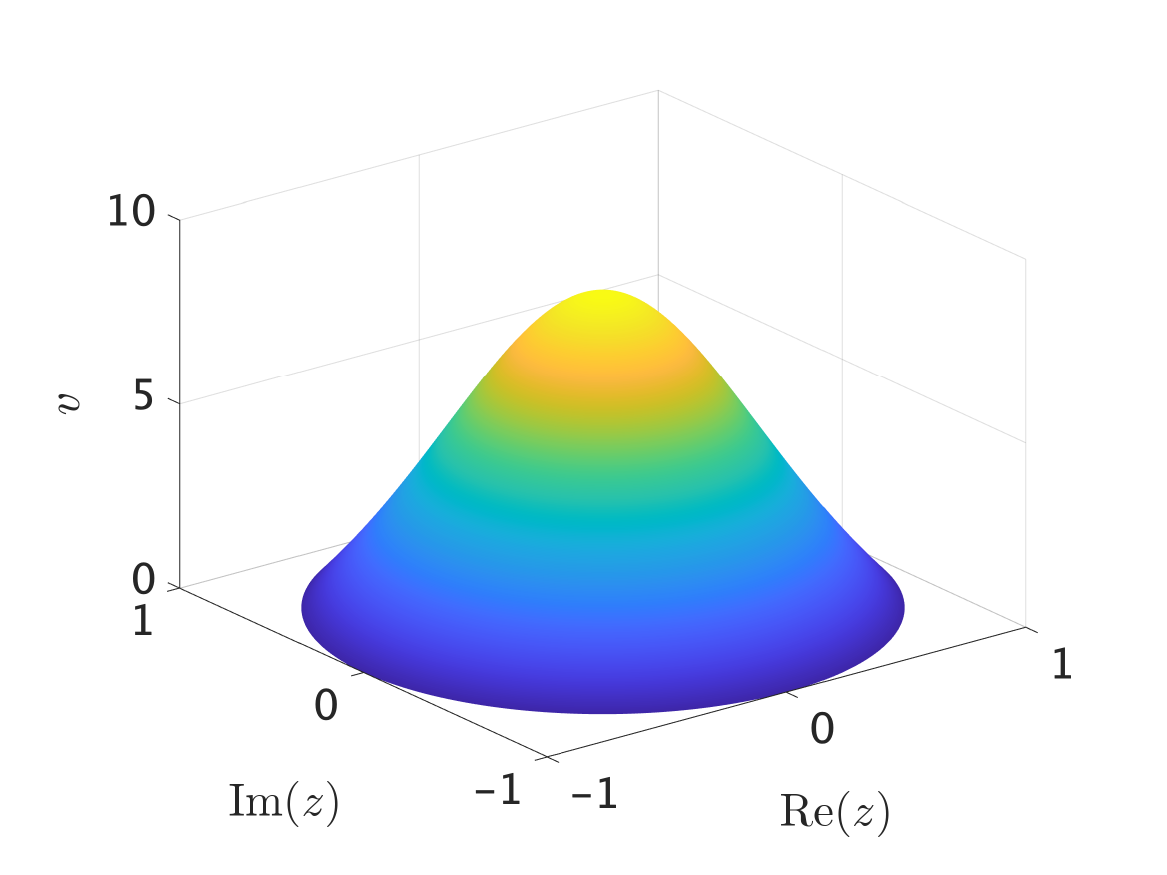}
	\end{minipage}
	\begin{minipage}[b]{0.48\linewidth}
		\centering
		\includegraphics[width=\textwidth]{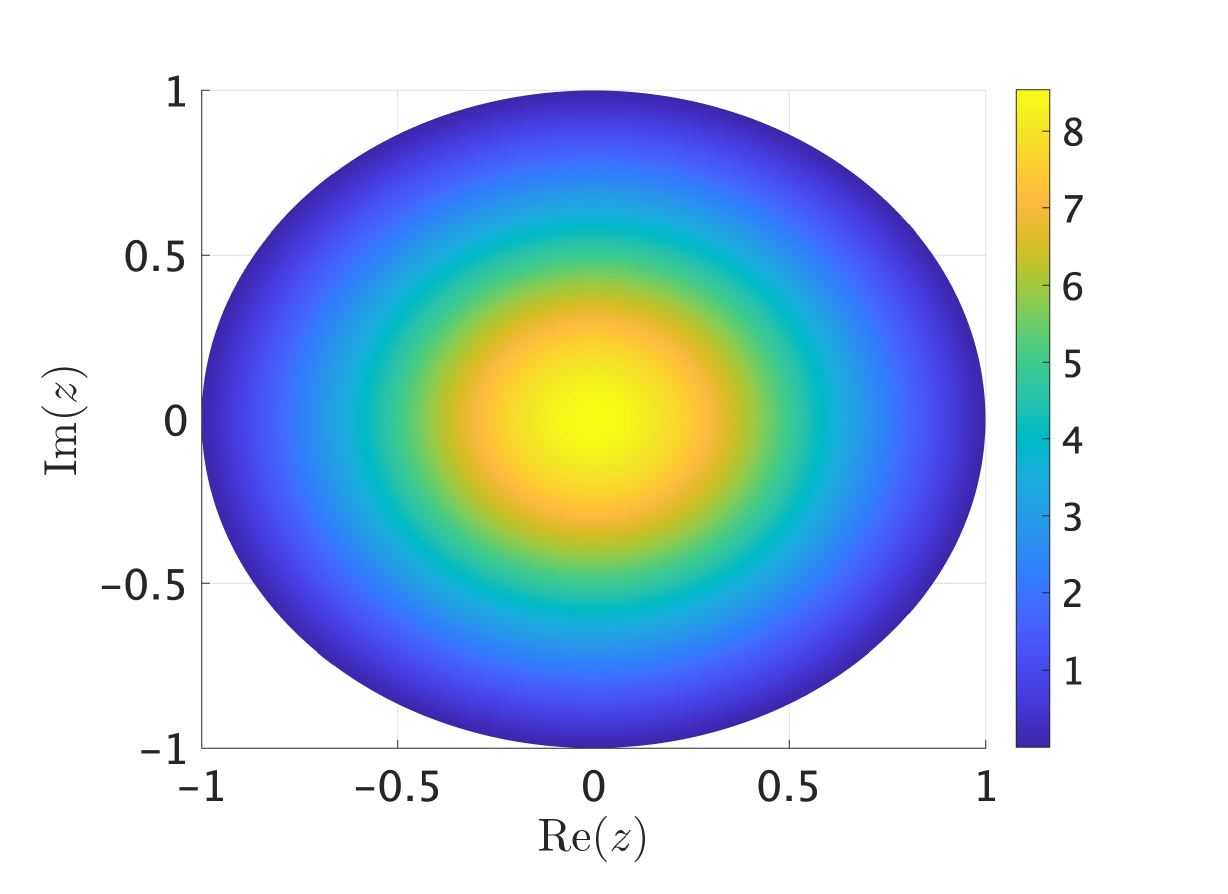}
	\end{minipage}
	\caption{Real part of the approximate solution with $m=0$.}\label{fig:4}
\end{figure}

\begin{figure}[htbp]\centering
	\begin{minipage}[b]{0.48\linewidth}
	\centering
        \includegraphics[width=\textwidth]{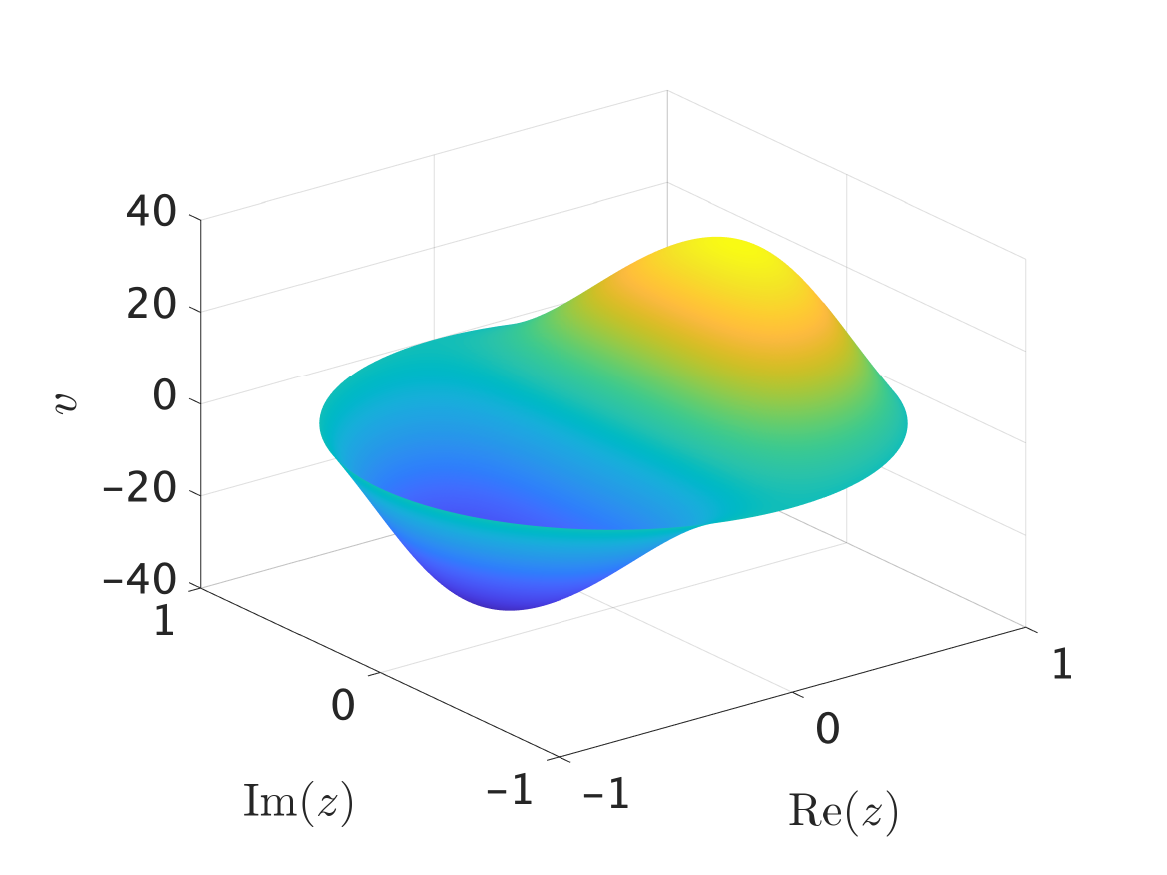}
	\end{minipage}
	\begin{minipage}[b]{0.48\linewidth}
	\centering
	\includegraphics[width=\textwidth]{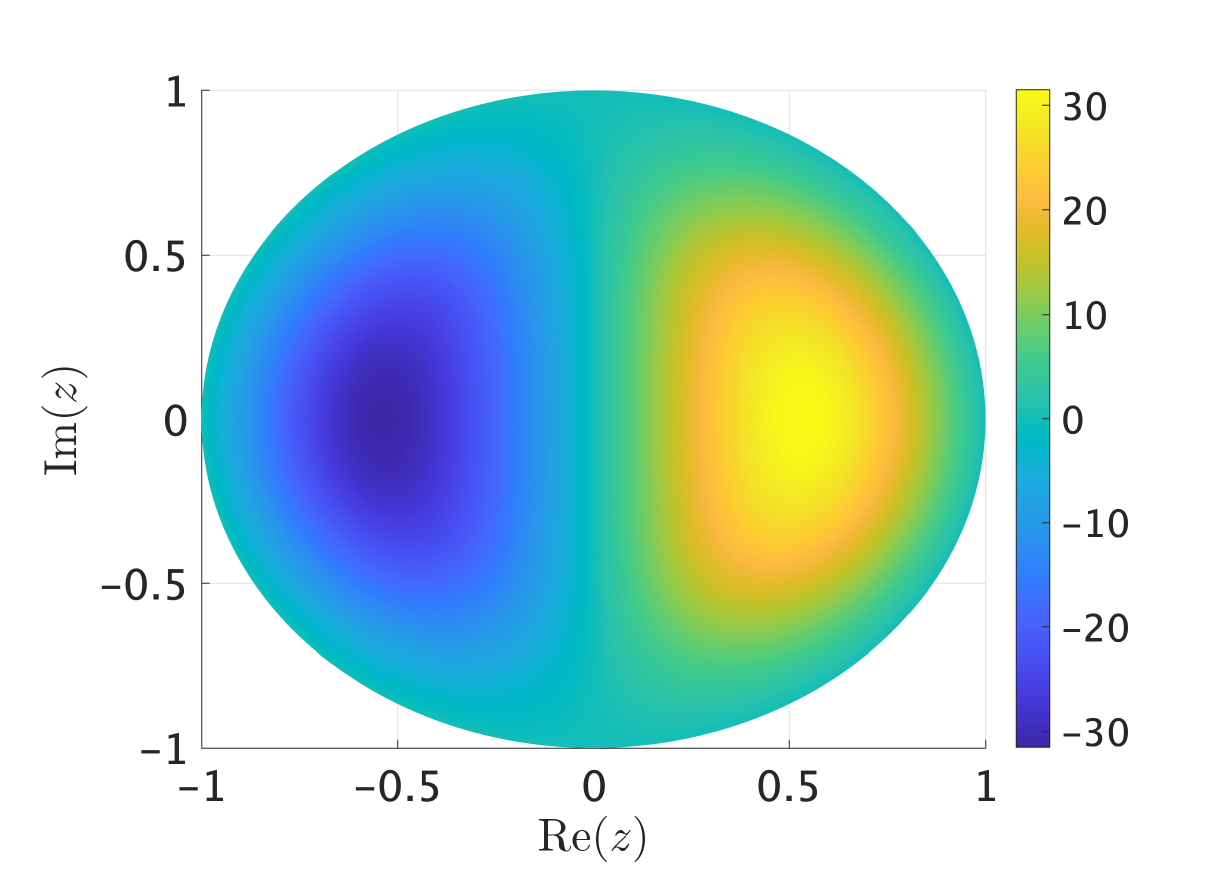}
	\end{minipage}
	\caption{Real part of the approximate solution with $m=1$.}\label{fig:5}
\end{figure}

\begin{figure}[htbp]\centering
	\begin{minipage}[b]{0.48\linewidth}
	\centering
        \includegraphics[width=\textwidth]{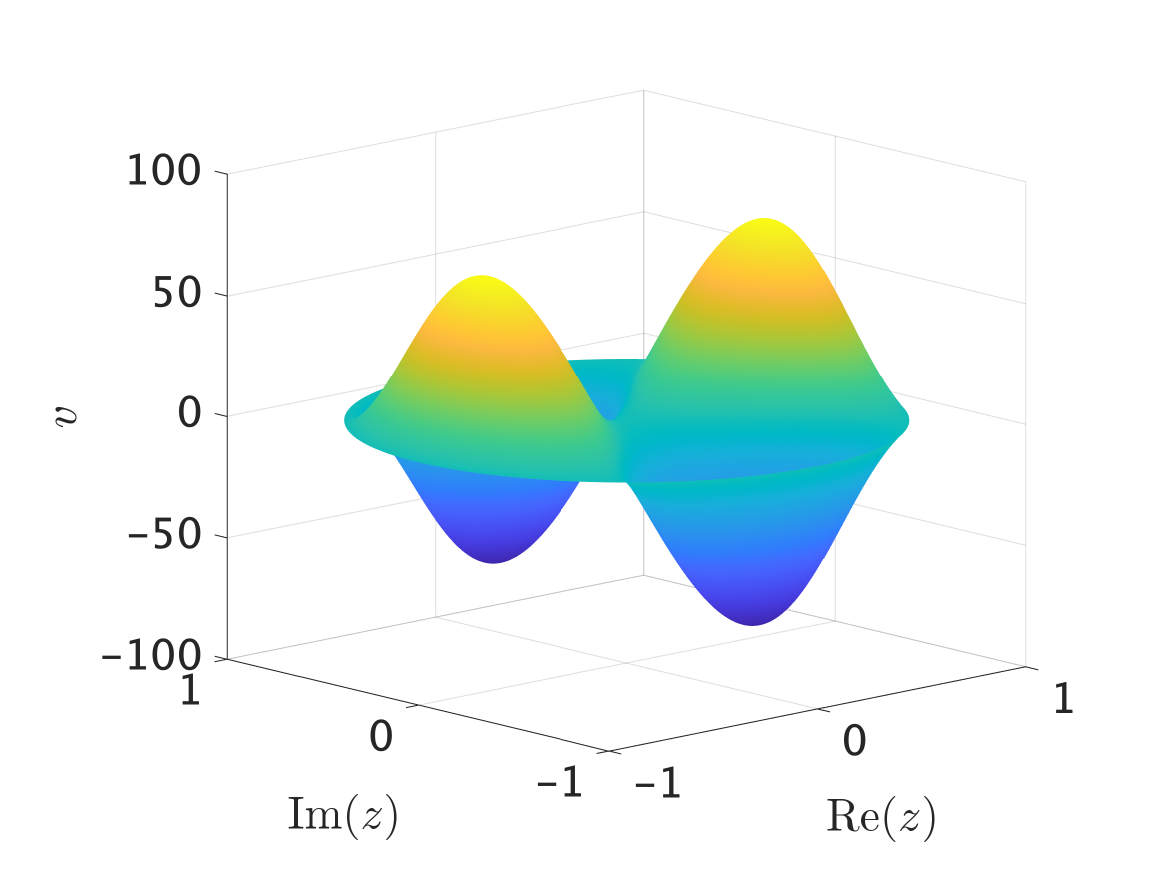}
	\end{minipage}
	\begin{minipage}[b]{0.48\linewidth}
	\centering
	\includegraphics[width=\textwidth]{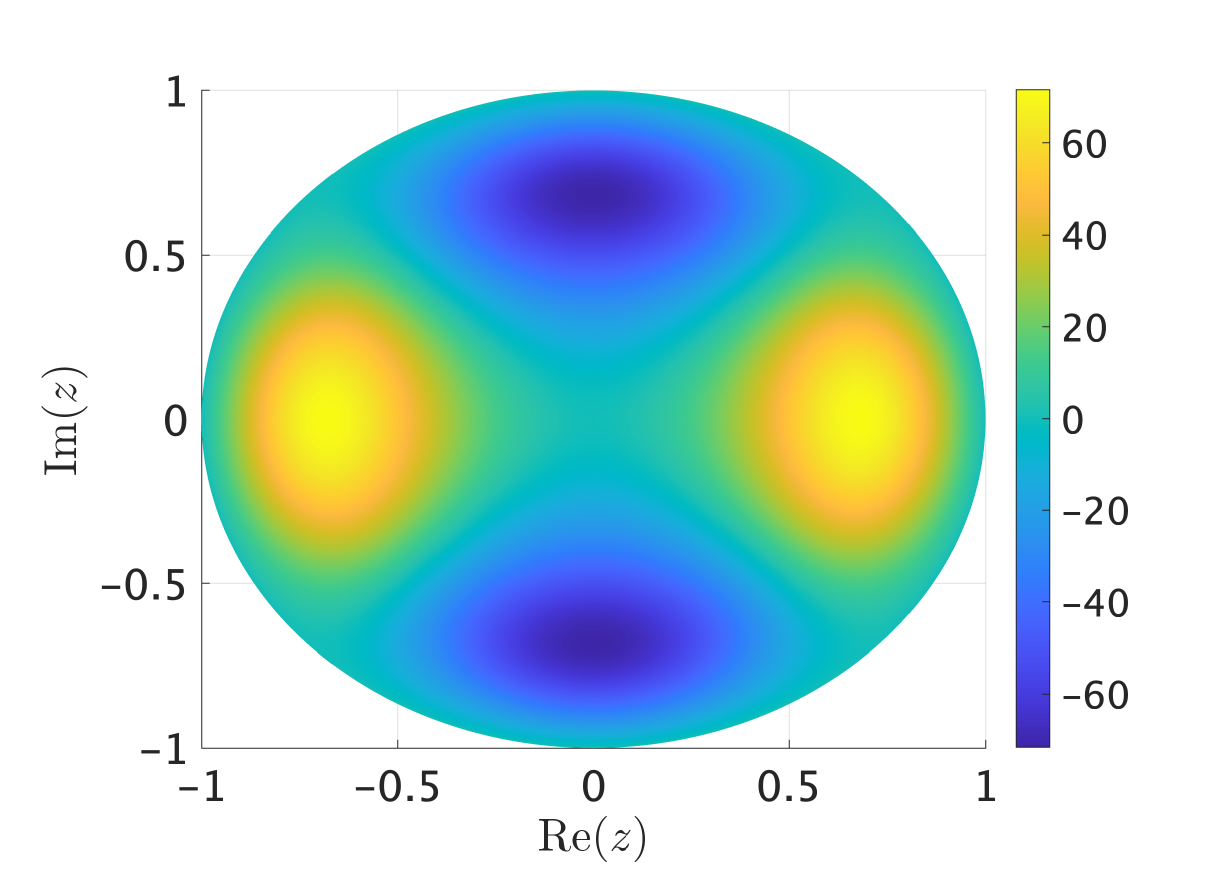}
	\end{minipage}
	\caption{Real part of the approximate solution with $m=2$.}\label{fig:6}
\end{figure}

\begin{figure}[htbp]\centering
	\begin{minipage}[b]{0.48\linewidth}
	\centering
        \includegraphics[width=\textwidth]{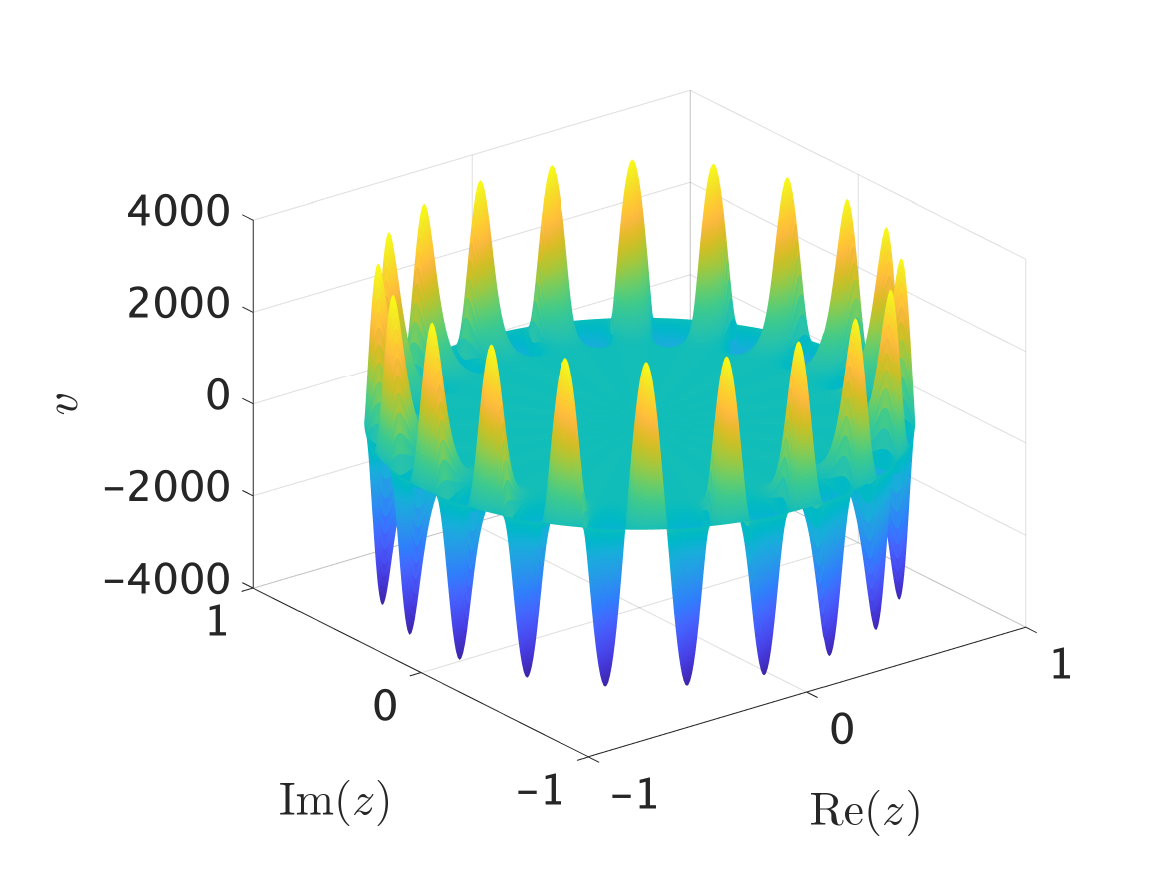}
	\end{minipage}
	\begin{minipage}[b]{0.48\linewidth}
	\centering
	\includegraphics[width=\textwidth]{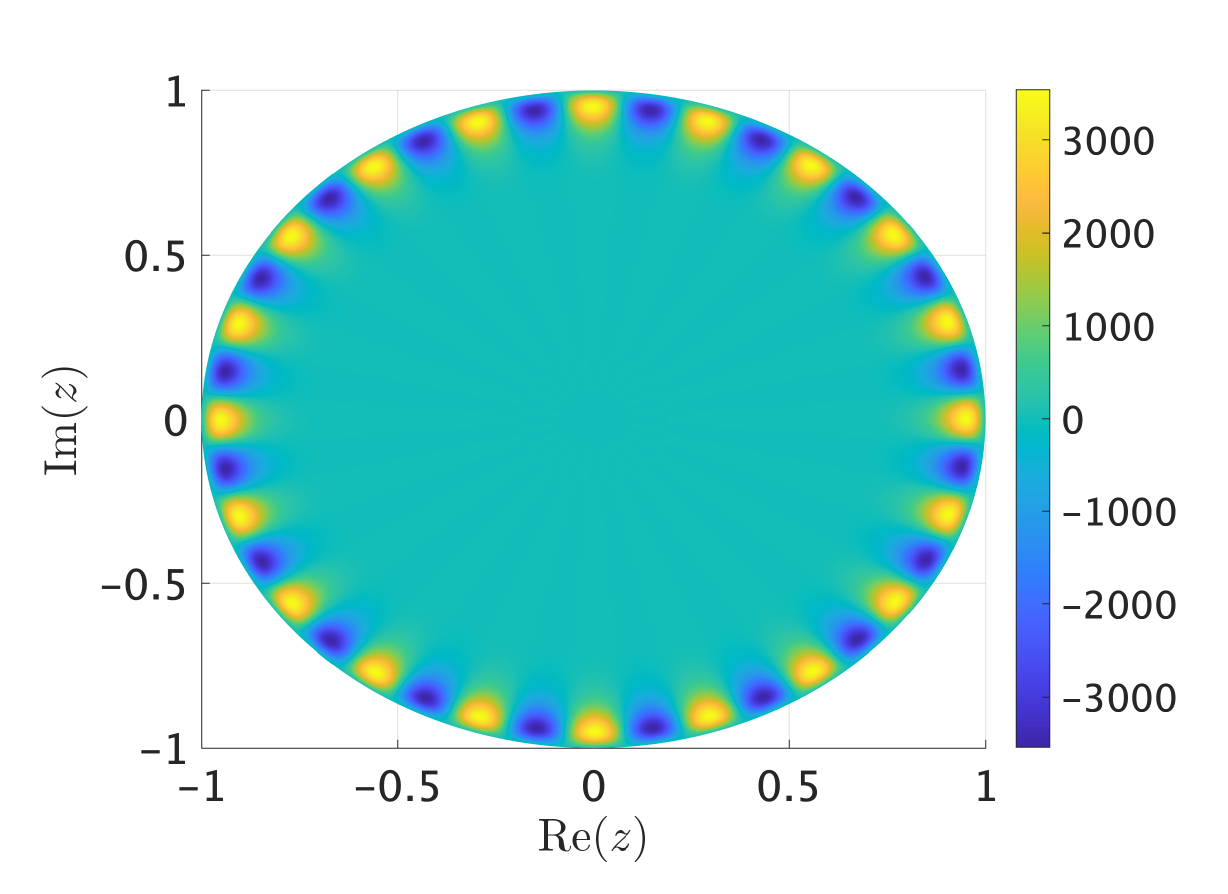}
	\end{minipage}
	\caption{Real part of the approximate solution with $m=20$.}\label{fig:7}
\end{figure}

\appendix

\section{Some operator Bounds}\label{sec : appendix some operator bounds}

Recalling the definition of the operator $R^-$ in \eqref{def : definition of R plus and R minus on jacobi}, we can define the multiplication operator by $z^m$ on $V^{0,2m} \to V^{0,m}$, that we denote $(R^{-})^m$, as follows
\begin{align}\label{def : multiplication by r^m}
(R^{-})^m \colon V^{0,2m} \overset{R^-}\to V^{0,2m-1}  \overset{R^-}\to V^{0,2m-2}  \dots   \overset{R^-}\to  V^{0,m+1} \overset{R^-}\to V^{0,m}
\end{align}
Using that $R^{-}$ is lower triangular, we readily have
\begin{align*}
    \|R^{-}\|_1 = \max_{m \in \mathbb{Z}}\max_{n \in \mathbb{N}_0}\left\{ \frac{n+1}{2n+m+1} + \frac{n+m}{2n+m+1} \right\} = 1.
\end{align*}
Consequently, this implies that 
\begin{align}\label{eq : op norm rm}
    \|(R^{-})^m\|_1 \leq 1.
\end{align}

Now, recall that we define $R^+_{0,1} : V^{0,1} \to V^{0,2}$ (resp. $R^+_{0,0} : V^{0,0} \to V^{0,1}$)  as the restriction of $R^+$ (cf. \eqref{def : definition of R plus and R minus on jacobi}) to $V^{0,1} \to V^{0,2}$ (resp. $V^{0,0} \to V^{0,1}$). Consequently, both  $R^+_{0,0}$ and  $R^{+}_{0,1}$ can be represented as upper-triangular linear operators on sequences indexed on $\mathbb{N}_0$. The following lemma provides an explicit inverse for $R^+_{0,0}$ and  $R^{+}_{0,1}$ which will be used for computing the estimates in Section \ref{sec : computer assisted analysis}.

\begin{lemma}\label{lem : inverse of R01}
    The operator $R^+_{0,1} : V^{0,1}_{s} \to V^{0,2}_{s-1}$ is invertible and its inverse, denoted $(R^+_{0,1})^{-1} : V^{0,2}_s \to V^{0,1}_{s-1}$, is given by
    \begin{align}\label{def : 1/r}
    \left((R^+_{0,1})^{-1}\right)_{i,j} = \begin{cases}
        (-1)^{j+i}\frac{2(i+1)^2}{(j+1)(j+2)}  &\text{ if } j \geq i\\
          0  & \text{ otherwise}.        
    \end{cases}
\end{align}
Moreover, 
\begin{align}\label{eq : norm identity inv R}
    \|(R^{+}_{0,1})^{-1}\|_{B(V^{0,2}_0, V^{0,1}_{-1})}  \leq  \frac{1}{2},
\end{align}
where $\|\cdot\|_{B(V^{0,2}_0, V^{0,1}_{-1})}$ denotes the operator norm for bounded linear  operators on $V^{0,2}_{0} \to V^{0,1}_{-1}$.
Similarly, $R^+_{0,0} : V^{0,0}_s \to V^{0,1}_{s-1}$ is invertible and its inverse, denoted $(R^+_{0,0})^{-1} : V^{0,1}_s \to V^{0,0}_{s-1}$, is given by
    \begin{align}\label{def : 1/r 0}
    \left((R^+_{0,0})^{-1}\right)_{i,j} = \begin{cases}
        (-1)^{j+i}\frac{2i+1}{j+1}  &\text{ if } j \geq i\\
          0  & \text{ otherwise}.        
    \end{cases}
\end{align}
\end{lemma}

\begin{proof}
Let $f = (f_i)_{i \in \mathbb{N}_0} \in V^{0,2}$.
First recall from \eqref{def : definition of R plus and R minus on jacobi} that  
$R^+_{0,m} \cQ_i^{0,m}  =  
	\frac{i+|m|+1}{2i+|m|+1} \cQ_{i}^{0,m+1} + \frac{i}{2i+|m|+1} \cQ_{i-1}^{0,m+1}  $.
Suppose that we have define $\alpha_i = \frac{i+2}{2(i+1)}$ and $\beta_{i} = \frac{i}{2i+2}$, then, in order to invert $R^+_{0,1}$,  we need to find $x = (x_i)_{i \in \mathbb{N}_0}$ satisfying
\begin{align*}
       \alpha_i x_{i} + \beta_{i+1} x_{i+1} = f_i
\end{align*}
for all $i\in \mathbb{N}_0$.
The above is equivalent to
\begin{align}\label{eq : recurrence formula inverse}
    x_i = \frac{f_i}{\alpha_i} - \frac{\beta_{i+1}}{\alpha_{i}}x_{i+1}
    = \frac{f_i}{\alpha_i} - \frac{\beta_{i+1}}{\alpha_{i}}\left(\frac{f_{i+1}}{\alpha_{i+1}} - \frac{\beta_{i+2}}{\alpha_{i+1}}x_{i+2}\right)
    = \sum_{j=i}^\infty (-1)^{i+j}\frac{\beta_{i+1}\dots \beta_{j}}{\alpha_i \dots \alpha_j} f_j 
\end{align}
since  $f_j \to 0$ as $j \to \infty.$ In particular, we have 
\begin{align*}
    \frac{\beta_j}{\alpha_j} = \frac{j}{j+2}
\end{align*}
and therefore,
\begin{align*}
    \frac{\beta_{i+1}\dots \beta_{j}}{\alpha_i \dots \alpha_j} =  \frac{1}{\alpha_i}
    \frac{\prod_{n=i+1}^j n}{\prod_{n=i+1}^j (n+2)}
    =
       \frac{2(i+1)^2}{(j+1)(j+2)} 
\end{align*}
for all $j \geq i.$ This provides the desired formula for the entries of $(R^+_{0,1})^{-1}$. Finally, using Definition \ref{def : Vk,m,s} with $\omega_{m,n} = (1+2n+|m|)^s$, we have
\begin{align}\label{eq : appendix lemma}
    \|(R_{0,1}^+)^{-1}\|_{B(V^{0,2}_s, V^{0,1}_{s-1})}= 
    \max_{j \in \mathbb{N}_0} \left\{ \frac{1}{(2j+3)^s } \sum_{i=0}^j \frac{2(i+1)^2}{(j+1)(j+2)} (2i+2)^{s-1}    \right\} 
\end{align}
As the sum is monotonic, we may bound it above by an integral, from which we may see that $\|(R_{0,1}^+)^{-1}\|_{B(V^{0,2}_s, V^{0,1}_{s-1})} < \infty$ for all $s \geq 0$. 
For calculating the norm with $ s=0$ this reduces to:
\begin{align*}
    \|(R_{0,1}^+)^{-1}\|_{B(V^{0,2}_0, V^{0,1}_{-1})}= 
    \max_{j \in \mathbb{N}_0}\left\{ \sum_{i=0}^j\frac{(i+1)}{(j+1)(j+2)}  \right\} = \frac{1}{2}.
\end{align*}
In order to compute the inverse of $R^+_{0,0}$, we use the above reasoning with $\alpha_i = \frac{i+1}{2i+1}$ and $\beta_{i} = \frac{i}{2i+1}$ and obtain the entries of the inverse using \eqref{eq : recurrence formula inverse}.
\end{proof}

\paragraph{Acknowledgement}
This work was conducted during the thematic semester titled ``Computational Dynamics - Analysis, Topology \& Data'', supported by the Centre de Recherches Math\'ematiques at the University of Montreal and the Simons Foundation. The authors sincerely thank these institutions for their funding and for providing the opportunity to explore this research.
AT is also partially supported by the Top Runners in Strategy of Transborder Advanced Researches (TRiSTAR) program conducted as the Strategic Professional Development Program for Young Researchers by the MEXT, and by JSPS KAKENHI Grant Numbers JP22K03411, JP21H01001, and JP24K00538.

\paragraph{Declaration of generative AI and AI-assisted technologies in the writing process}
During the preparation of this work, the authors used ChatGPT to assist in refining the English wording and enhancing the clarity of the presented content. After using this tool/service, the authors reviewed and edited the content as needed and take full responsibility for the content of the publication.

\bibliography{biblio}
\bibliographystyle{alpha}

\end{document}